\documentclass[%
reprint%
,onecolumn%
,jmp%
,a4paper,superscriptaddress,nofootinbib,showkeys]{revtex4-2}

\usepackage{mathrsfs}

\usepackage{amsmath,amssymb,bm,cancel}
\usepackage{amsthm}
\usepackage{mathtools}

\newtheorem{theorem}{Theorem}[section]
\newtheorem{proposition}{Proposition}[section]

\newtheorem{corollary}{Corollary}[section]

\theoremstyle{definition}
\newtheorem{definition}{Definition}[section]

\theoremstyle{remark}

\newtheorem{example}{Example}[section]

\usepackage{graphicx}
\usepackage[dvipsnames]{xcolor}
\usepackage{tikz}
\usetikzlibrary{arrows.meta}

\begin{document}

\setlength\parindent{0pt}
\setlength{\parskip}{0.5em}

\title{Linearisability of divergence-free fields along invariant 2-tori}

\author{David Perrella}
\author{David Pfefferl{\'e}}
\author{Luchezar Stoyanov}
\affiliation{The University of Western Australia, 35 Stirling Highway, Crawley WA 6009, Australia}

\begin{abstract}
    We find conditions under which the restriction of a divergence-free vector field $B$ to an invariant toroidal surface $S$ is linearisable. The main results are similar in conclusion to Arnold's Structure Theorems but require weaker assumptions than the commutation $[B,\nabla\times B] = 0$. Relaxing the need for a first integral of $B$ (also known as a flux function), we assume the existence of a solution $u : S \to \mathbb{R}$ to the cohomological equation $B|_S(u) = \partial_n B$ on a toroidal surface $S$ mutually invariant to $B$ and $\nabla \times B$. The right hand side $\partial_n B$ is a normal surface derivative available to vector fields tangent to $S$. In this situation, we show that the field $B$ on $S$ is either identically zero or nowhere vanishing with $B|_S/\|B\|^2 |_S$ being linearisable. We are calling the latter the semi-linearisability of $B$ (with proportionality $\|B\|^2 |_S$). The non-vanishing property relies on Bers' results in pseudo-analytic function theory about a generalised Laplace-Beltrami equation arising from Witten cohomology deformation. With the use of de Rham cohomology, we also point out a Diophantine integral condition where one can conclude that $B|_S$ itself is linearisable. The linearisability of $B|_S$ is fundamental to the so-called magnetic coordinates, which are central to the theory of magnetically confined plasmas.
\end{abstract}

\maketitle

\section{Introduction}

The linearisability of a vector field on a 2-torus is a well understood dynamical problem~\cite{SternbergCelestial,KocsardCohomological}. A vector field $X$ on a 2-torus $S$ is said to be \emph{linearisable} if there exists numbers $a,b \in \mathbb{R}$ and a diffeomorphism $\Phi : S \to \mathbb{R}^2/\mathbb{Z}^2$ such that $\Phi_* X$ is a constant vector field on $\mathbb{R}^2$ lowered to $\mathbb{R}^2/\mathbb{Z}^2$. That is,
\begin{equation}\label{lineq}
\Phi_* X =  a \frac{\partial}{\partial x} + b \frac{\partial}{\partial y}.
\end{equation}

A slightly weaker condition for $X$ is the existence of a positive function $f$ on $S$ for which $X/f$ is linearisable. We will say that such an $X$ is \emph{semi-linearisable with proportionality $f$}. This is equivalent to the existence of a coordinate system in which the field-lines of $X$ are straight~\cite{SternbergCelestial}. The field-lines of $X$ are known as windings in such a coordinate system~\cite{arnold1974asymptotic}.

This paper mainly concerns these properties when $X$ is the vector field induced on an invariant 2-torus of a divergence-free vector field in an oriented Riemannian 3-manifold $M$ with (or without) boundary. The motivation lies in producing and understanding magnetohydrodynamics (MHD) equilibria~\cite{kruskal_1958,hamada-1962,dhaeseleer_1991,etnyre-ghrist}, which are solutions to the system of equations
\begin{align*}
\nabla \cdot B &= 0,&
(\nabla \times B) \times B &= \nabla \rho,&
 B \cdot n &= 0 \text{ on } \partial M,
\end{align*}
where $\rho$ is a function on $M$ interpreted as pressure, $B$ is the magnetic field in $M$, and $n$ is the outward unit normal of $\partial M$.

The conjecture of Grad~\cite{grad1967toroidal,bruno-laurence_1996} remains unsettled; that smooth solutions with $p$ admitting toroidally nested level sets only exist if $M$ has a continuous isometry. One way to better understand this conjecture is to prove necessary structural features of solutions if they exist. A contribution by Arnold in this regard is his structure theorems~\cite{arnold1966topology,arnold-1966,arnold1974asymptotic}. In particular, Arnold obtained the following linearisablility result.

\begin{theorem}[\citet{arnold1974asymptotic}]
Suppose that $\rho$ has a closed (compact, without boundary) regular level set. Then, the connected components of this level set are invariant tori of $B$ and $\nabla \times B$ and in some neighbourhood $U\subset M$ of such a component, there exist a diffeomorphism $\Phi : U \to \mathbb{R}^2/\mathbb{Z}^2 \times I$ where $I \subset \mathbb{R}$ is an interval, such that
\begin{align*}
\Phi_* B|_U &= a(z) \frac{\partial}{\partial x} + b(z) \frac{\partial}{\partial y},&
\Phi_* \nabla \times B |_U &= c(z) \frac{\partial}{\partial x} + d(z) \frac{\partial}{\partial y},
\end{align*}
where $a,b,c,d : I \to \mathbb{R}$ are smooth functions and $z$ is the projection onto the factor $I$.
\end{theorem}
Coordinates in which both $B$ and $\nabla\times B$ are linear are known as Hamada coordinates~\cite{hamada-1962}. In magnetic confinement fusion, coordinates systems in which $B$ is linear (regardless of $\nabla \times B$) are called "straight field-lines" coordinates~\cite{dhaeseleer_1991,boozer-1982}.

Arnold also remarked~\cite{arnold1974asymptotic} that in the case of $\nabla \rho = 0$, assuming $B$ is non-vanishing, then necessarily
\begin{itemize}
    \item[i)] $B$ is a Beltrami field, namely $\nabla \times B = \lambda B$ for some function $\lambda$, 
    \item[ii)] $\lambda$ is a first integral and the closed regular level sets of $\lambda$ are unions of tori, and
    \item[iii)] in a neighborhood $U\subset M$ of such a 2-torus, there exists a diffeomorphism $\Phi : U \to \mathbb{R}^2/\mathbb{Z}^2 \times I$ and a positive function $f : U \to \mathbb{R}$ such that,
    \begin{align*}
    \Phi_* (B|_U/f) &= a(z) \frac{\partial}{\partial x} + b(z) \frac{\partial}{\partial y}
    \end{align*}
    where again $a,b: I \to \mathbb{R}$ are smooth functions and $z$ is the projection onto the factor $I$.
\end{itemize}

In this paper, we will bring some more attention to this remark. More specifically, we show that it is not particularly a consequence of $B$ being Beltrami field or even that of the non-vanishing of $B$. For instance, if one assumes a first integral (we will also state an ``infinitesimal version" which features no first integral assumption), we have the following corollary of our methods.

\begin{corollary}\label{GeneralArnold3D}
Let $B$ be a vector field on $M$ which satisfies, for some function $\rho$,
\begin{align*}
\nabla \cdot B &= 0,& B \cdot \nabla \rho &= 0,& (\nabla \times B) \cdot \nabla \rho &= 0.
\end{align*}
Let $S$ be a closed connected component of a regular level set of $\rho$. Then, the following are equivalent
\begin{enumerate}
    \item $B$ is not identically zero on $S$ and $S$ is a 2-torus
    \item $B$ is non-vanishing on $S$
\end{enumerate}
and when this is the case, additionally assuming $S \cap \partial M = \emptyset$ (or $S \subset \partial M$), there exists a neighbourhood $U\subset M$ of $S$ and a diffeomorphism $\Phi : U \to \mathbb{R}^2/\mathbb{Z}^2 \times I$ where $I$ is an open (or half-open) interval in $\mathbb{R}$ such that $B|_U$ is non-vanishing and
\begin{equation*}
\Phi_* \left(\frac{B|_U}{\|B\|^2 |_U}\right) = a(z) \frac{\partial}{\partial x} + b(z) \frac{\partial}{\partial y}
\end{equation*}
where $a,b : I \to \mathbb{R}$ are smooth functions and $z$ is the projection onto the factor $I$.
\end{corollary}

A strong Beltrami field $B$ satisfying $\nabla \times B = \lambda B$ where $\lambda$ is a constant, can have more complicated topology~\cite{arnold1974asymptotic,etnyre-ghrist} despite being MHD equilibria solutions~\cite{bruno-laurence_1996,dewar_2008}. Less has been said about the structure (in relation to linearisability) of strong Beltrami fields on invariant tori. In particular, if $\partial M$ has a toroidal connected component and $B \cdot n = 0$, this component is an invariant 2-torus of $B$ and it is reasonable to ask of the linearisability properties of $B$ here. This question has been linked with the aforementioned conjecture of Grad. In the Euclidean context,~\citet{enciso2021mhd} have shown that under non-degeneracy assumptions of a toroidial domain $M \subset \mathbb{R}^3$, piece-wise smooth MHD equilibria with non-constant pressure exist. The non-degeneracy assumptions include assuming the existence of a strong Beltrami field $B$ for which $\partial M$ is a Diophantine invariant 2-torus. That is, as in the context of the KAM Theorem, the vector field $X$ (induced from $B$ on $\partial M$) can be written in the form of Equation (\ref{lineq}) where the vector $(a,b)$, known as \emph{the frequency vector of $X$ with respect to $\Phi$}, is a Diophantine vector (a notion which will be defined in Section \ref{sec:mainresdef}). Of course, the frequency being Diophantine does not depend on the diffeomorphism $\Phi$ chosen (see, for instance, Propositions \ref{correctnessrotationaltransform} and \ref{compatibilityrotationaltransform}). The authors managed to show that the so-called thin toroidal domains are generically non-degenerate~\cite{enciso2021mhd}. In view of Corollary \ref{GeneralArnold3D}, the existence of a first integral of $B$ will ensure some structure of $B$. Although, in face of this complicated topology, there is no reason to expect that a first integral exists.

As mentioned, the ``infinitesimal version" of Corollary \ref{GeneralArnold3D} does not assume a first integral. We will now state this version in the case most relevant to strong Beltrami fields; namely when $M$ is embedded in $\mathbb{R}^3$ and $S$ is a toroidal (connected) component of the boundary $\partial M$. For instance, relevant to the Stepped Pressure Equilibrium Code ~\cite{hudson2012computation} (SPEC, a program which numerically solves for MHD equilibria), $S$ could be taken as either component of the boundary $\partial M$ when $M$ is a \emph{hollow torus}, that is, when $M$ is diffeomorphic to $\mathbb{R}^2/\mathbb{Z}^2 \times [0,1]$.

\begin{corollary}\label{advertcorollary}
Let $M$ be embedded in $\mathbb{R}^3$ with the inherited Euclidean structure. Let $S$ with a toroidal boundary component of $M$. Let $n : \partial M \to \mathbb{R}^3$ be the outward unit normal on $M$. Let $B$ be a vector field on $M$ satisfying, for some $\lambda \in \mathbb{R}$,
\begin{align*}
\nabla \cdot B &= 0,&
\nabla \times B &= \lambda B,&
B|_{\partial M} \cdot n &= 0.
\end{align*}
Consider $B|_{S}$, the vector field $B$ along $S$. Then, the following are equivalent.
\begin{enumerate}
    \item There exists a solution $u \in C^{\infty}(S)$ to the cohomological equation,
    \begin{equation*}
    B|_{S}(u) = \partial_n B.
    \end{equation*}
    \item The vector field $B|_{S}$ preserves a top-form $\mu$ on $S$.
    \item Either $B = 0$, or $B$ is non-vanishing on $S$ and $B|_{S}$ is semi-linearisable with proportionality $\|B\|^2|_{S}$.
\end{enumerate}
With such a solution $u$ and closed curves $C_1,C_2 : [0,1] \to S \subset \mathbb{R}^3$ whose homology classes generate the first homology $H_1(S)$, form the integrals
\begin{equation*}
I(C_i,u) \coloneqq \int_{0}^1 \exp{(-u(C_i(t)))} \det(C_i'(t),B(C_i(t)),n(C_i(t)))dt,~i \in \{1,2 \}.
\end{equation*}
If $(I(C_1,u),I(C_2,u))$ is Diophantine, then $S$ is a Diophantine invariant 2-torus for $B|_{S}$. The vector $(I(C_1,u),I(C_2,u))$ being Diophantine is independent of the solution $u$ and curves $C_1,C_2$ chosen.
\end{corollary}
In Corollary \ref{advertcorollary}, if the field is zero on the boundary component $S$, $B|_{S} = 0$, then from Gerner's result~\cite[Lemma 2.1]{gerner2021typical} we must have $B = 0$ on the entire domain $M$ (see also Proposition \ref{2dimBeltrami}). The function $\partial_n B$ in Corollary \ref{advertcorollary} and elsewhere will be defined in Section \ref{sec:mainresdef} of the main text. It represents a derivative in the normal direction which may be evaluated using any local extension of $n$. In particular, if a first integral $\rho$ of $B$ is constant and regular on $S$, then $u = -\ln\|\nabla \rho\||_{S}$ satisfies the cohomological equation in the corollary (this is Proposition \ref{for first integrals} in Section \ref{sec:mainresdef}).

Aside from the 2-torus, the techniques used to show the main results may also be easily applied to the 2-sphere. This gives a similar result to that obtained by~\citet{enciso2016beltrami} for divergence-free vector fields which are not necessarily Beltrami.

\begin{corollary}\label{divfreesphere}
Let $B$ be a divergence-free vector field on an oriented Riemannian 3-manifold $M$ with boundary. Suppose that $B$ and $\nabla \times B$ have a mutual first integral $\rho$ on $M$ with a connected component $S$ of a regular level set diffeomorphic to $\mathbb{S}^2$. Then, $B$ vanishes entirely in a neighbourhood of $S$.
\end{corollary}

Our other results about vector fields on a 2-torus will mostly be stated in terms of their winding number (or frequency ratios, as in \cite{arnold1974asymptotic}). We have employed a homology-dependent means of defining winding numbers, so that by way of de Rham cohomology, we may compute them with integrals like those in Corollary \ref{advertcorollary}. In the context of magnetic confinement~\cite{boozer2005physics,d2012flux}, the winding number corresponds to what is known as the rotational transform. The rotational transform plays an integral role in stability of magnetically confined plasmas~\cite{wesson2011tokamaks}. In a future paper, we will relate known rotational transform formulae to what is presented here.

This paper is structured as follows. First, in Section \ref{sec:mainresdef}, we state the main results with some preliminary definitions. In Section \ref{sec:proofofres}, we prove these results using Witten-deformed cohomology~\cite{witten1982supersymmetry} and the elliptic PDE theory developed by Bers~\cite{bers1953theory,bers1953partial}. In Section \ref{sec:exandapp}, we give some applications and examples including proofs of Corollaries \ref{GeneralArnold3D}, \ref{advertcorollary} and \ref{divfreesphere}. In Section \ref{sec:RigidityWitten}, we discuss our use of the theory of Bers and how, in a certain sense, this generality is needed for the full result. In Appendices \ref{app:normcorrect} and \ref{app:rottranscorrect}, we establish correctness of the definitions. In Appendix \ref{app:rottranslinear}, we discuss some foundational properties of the winding number in relation to linearisability.

\section{Main results and definitions}\label{sec:mainresdef}

In this section, we state the main results. For this, we will need to first define the normal surface derivative mentioned in the Introduction. For terminology with smooth manifold theory, we follow Lee's book~\cite{Lee}. Unless otherwise stated, everything is assumed to be smooth for convenience.

\begin{definition}\label{normalderivative}
Let $M$ be a Riemannian manifold with boundary with metric $g$. Let $S$ be an orientable codimension $1$ embedded submanifold with (or without) boundary. Let $\mathcal{V} : S \to TM$ be a normal vector field to $S$, and $B:M\to TM$ be a vector field tangent to $S$. Define the function $\partial_{\mathcal{V}}B : S \to \mathbb{R}$ by,
\begin{equation*}
\partial_{\mathcal{V}}B|_p = g|_p([V,B]|_p,V|_p),
\end{equation*}
where $V : U \to TM$ is a local vector field extending $\mathcal{V}|_{U \cap S}$. This function is called the \emph{normal derivative of $B$ with respect to $\mathcal{V}$ along $S$}.
\end{definition}

We will show in Appendix \ref{app:normcorrect} that Definition \ref{normalderivative} is correct. With respect to this definition, the main part of our results giving Corollary \ref{advertcorollary} is the following.

\begin{theorem}\label{mainresult}
Let $M$ be an oriented Riemannian $3$-manifold with boundary. Let $S$ be an embedded 2-torus in $M$. Let $n : S \to TM$ be a unit normal for $S$. Let $B$ be a vector field on $M$ which satisfies,
\begin{align*}
\nabla \cdot B|_S &= 0,&
\nabla \times B|_S \cdot n &= 0,&
B|_S \cdot n &= 0.
\end{align*}
Consider the $\iota$-related vector field $\imath^*B$ on $S$ where $i : S \subset M$. If there exists a solution $u \in C^{\infty}(S)$ to the cohomological equation,
\begin{equation*}
du(\imath^*B) = \partial_n B,
\end{equation*}
then either $B|_S = 0$ or $B|_S$ is non-vanishing and $\imath^*B$ is semi-linearisable with proportionality $\|B\|^2|_S$.
\end{theorem}

A particular application of Theorem \ref{mainresult} is granted in the special context of first integrals as follows.

\begin{proposition}\label{for first integrals}
Let $M$ be an oriented Riemannian $3$-manifold with boundary. Let $S$ be a codimension 1 embedded submanifold with boundary in $M$. Let $B$ be a vector field on $M$ which satisfies $\nabla \cdot B|_S = 0$. Assume that $B$ has a first integral $\rho \in C^{\infty}(M)$ which is constant and regular on $S$. Then, $B|_S \cdot n = 0$ and, setting $u = -\ln\|\nabla \rho\||_S$,
\begin{equation*}
du(\imath^*B) = \partial_n B.
\end{equation*}
\end{proposition}

Theorem \ref{mainresult} is proven by the fact that $\imath^*B$ is a $P$-harmonic vector field on the 2-torus. We will now discuss the definitions and main results concerning these fields.

Let $S$ be an oriented Riemannian manifold with differential $d$ and codifferential $\delta = (-1)^{nk+n+1}\star d \star$ where $\star$ is the Hodge star operator acting on $k$-forms and $n = \dim S$. Let $0 < P \in C^{\infty}(S)$. A $k$-form $\omega \in \Omega^k(S)$ is called \emph{$P$-harmonic} if,
\begin{align*}
d\omega &= 0,&
\delta P \omega &= 0.
\end{align*}
Accordingly, a vector field $X$ on $S$ is called \emph{$P$-harmonic} if its flat (metric dual 1-form) $X^{\flat}\in \Omega^1(S)$ is a $P$-harmonic 1-form. Our result concerning $P$-harmonic 1-forms on the 2-torus is the following.

\begin{theorem}\label{P-harmonic 1-forms}
Let $S$ be an oriented Riemannian 2-torus. Let $\mathcal{H}_P^1(S)$ denote the real vector space of $P$-harmonic 1-forms on $S$. Then, the following holds.
\begin{enumerate}
    \item $\dim \mathcal{H}_P^1(S) = 2$.
    \item The map assigning a $P$-harmonic 1-form to its de Rham cohomology class,
    \begin{equation*}
        \mathcal{H}_P^1(S) \ni \omega \mapsto [\omega] \in H^1_{\text{dR}}(S)
    \end{equation*}
    is a linear isomorphism.
    \item If $\omega \in \mathcal{H}_P^1(S)$, then $\omega$ is either identically zero or non-vanishing.
\end{enumerate}
\end{theorem}

The cohomology class of a closed 1-form dual to a vector field is not particularly telling of the field-line topology on a oriented Riemannian 2-torus. This will be illustrated by an example in Section \ref{sec:exandapp}. In particular, the strictly dual result to Theorem \ref{P-harmonic 1-forms} does not address the field-line topology directly. Nevertheless, $P$-harmonic fields have a winding number, which is metric independent.

Instead of using fractions $a/b$ to define the winding number, we will use elements $[(a,b)]$ of the projective real line $\mathbb{P}(\mathbb{R}) = (\mathbb{R}^2\backslash \{0\})/{\sim}$ where vectors $u,v \in \mathbb{R}^2\backslash\{0\}$ are considered equivalent $u \sim v$ iff $u$ and $v$ are linearly dependent. This is to account for the case of $b = 0$ and to emphasise the connection with Diophantine vectors.

\begin{definition}\label{rotationaltransform}
Let $S$ be a 2-torus. Then, a densely-non-vanishing vector field $X$ is said be \emph{winding} if there exists a non-vanishing closed 1-form $\omega$ on $S$ such that $\omega(X) = 0$. Let $\gamma_1,\gamma_2 \in H_1(S)$ be generating classes of the first homology $H_1(S)$. Let $\omega$ be a closed 1-form with non-trivial cohomology class $[\omega] \neq 0$ such that $\omega(X) = 0$. Set
\begin{align*}
a &= -\int_{\gamma_2}\omega,& 
b &= \int_{\gamma_1}\omega.
\end{align*}
Then, $X$ is said to have \emph{Diophantine winding number} if the vector $(a,b)$ is Diophantine. The vector $(a,b) \in \mathbb{R}^2$ is non-zero and hence defines a class $[(a,b)] \in \mathbb{P(R)}$. The class $[(a,b)]$ is called the \emph{winding number of $X$ with respect to the generators $\gamma_1$ and $\gamma_2$}. 
\end{definition}

Here, a vector $u \in \mathbb{R}^2$ is said to be a \emph{Diophantine vector} if there exists $\gamma > 0$ and $\tau > 1$ such that, for all $k \in \mathbb{Z}^2\backslash \{0\}$, there holds $|\langle u,k \rangle| \geq \gamma \|k\|^{-\tau}$.

The proof of correctness of Definition \ref{rotationaltransform} is found in Appendix \ref{app:rottranscorrect} along with some compatibilities with other notions of winding number. Our result concerning $P$-harmonic vector fields on the 2-torus is the following.

\begin{theorem}\label{P-harmonic vector fields}
Let $S$ be an oriented Riemannian 2-torus. Let $0 < P \in C^{\infty}(S)$ and let $\mathcal{H}_P V(S)$ denote the real vector space of $P$-harmonic vector fields on $S$. Then, the following holds.
\begin{enumerate}
    \item $\dim \mathcal{H}_P V(S) = 2$.
    \item Let $X \in \mathcal{H}_P V(S)$. Then either $X$ is identically zero or $X$ is non-vanishing.
    \item If $X$ is non-vanishing, then $X$ is semi-linearisable with proportionality $\|X\|^2$. Moreover, $X$ is winding and if $X$ has Diophantine winding number, then $X$ is linearisable.
    \item Let $\gamma_1,\gamma_2$ be generating classes of the first homology of $S$. The winding number of $X$ with respect to $\gamma_1,\gamma_2$ is $[(a,b)]$ where,
    \begin{align*}
    a &= - \int_{\gamma_2}P\star X^{\flat},&
    b &= \int_{\gamma_1}P\star X^{\flat}.
    \end{align*}
    If $\gamma_1,\gamma_2$ are represented by closed curves, $C_i : [0,1] \to S$, $i \in \{1,2\}$, then
    \begin{equation*}
    \int_{\gamma_i}P\star X^{\flat} = \int_{0}^1 P(C_i(t)) \mu(X(C_i(t)),C_i'(t))dt,
    \end{equation*}
    where $\mu$ is the area element on $S$.
    \item For every $\tau \in \mathbb{P(R)}$, there exists a unique $Y \in \mathcal{H}_P V(S)$ (up to non-zero scalar multiplication) with winding number $\tau$ with respect to $\gamma_1,\gamma_2$.
\end{enumerate}
\end{theorem}

In the next section, we will prove these results. It should be emphasised early on that the less trivial part of Theorem \ref{P-harmonic vector fields} is the non-vanishing behaviour of solutions and dimension of the solution space. The remaining properties follow directly from the convenient algebraic structure of the equations. If one is willing to make non-vanishing assumptions from the beginning, as known to celestial mechanics~\cite{SternbergCelestial} the linearisability properties in Theorem \ref{P-harmonic vector fields} hold for vector fields satisfying more general equations. We formulate the result here and prove it in Appendix \ref{app:rottranslinear} using an alternative covariant approach.

\begin{theorem}\label{celestialmechanics}
Let $X$ be a non-vanishing vector field on a 2-torus $S$. Then, the following are equivalent.
\begin{enumerate}
    \item $X$ preserves a top-form $\mu$ on $S$.
    \item $X$ is winding.
    \item $X$ is semi-linearisable.
\end{enumerate}
Moreover, if $X$ has Diophantine winding number, then $X$ is linearisable.
\end{theorem}

\section{Proof of the main results}\label{sec:proofofres}

Theorem \ref{mainresult} is proven in Section \ref{subsec:P-harmonic vect} via Theorems \ref{P-harmonic 1-forms} and \ref{P-harmonic vector fields} and some basic computations. The proof of Theorem \ref{P-harmonic 1-forms} relies on Witten-deformed cohomology theory as well as Bers' pseudo-analytic function theory. We will introduce these theories in the course of the proof. The corresponding result for constant $P$ has a simplified proof which only relies on classical theory. We will discuss this in relation to the full result. After this, Theorem \ref{P-harmonic vector fields} follows from Theorem \ref{P-harmonic 1-forms} using the properties of the winding numbers and linearisability established in Appendix \ref{app:rottrans}. We will first present the basic computations required for the proof of Theorem \ref{mainresult}.

\subsection{Computations with codimension 1 submanifolds}

The computations we do here will be of higher generality than needed for Theorem \ref{mainresult} since no additional difficulty is met. For the following, let $M$ be an oriented Riemannian manifold with boundary, with metric $g$ and top form $\mu$. Let $S$ be a codimension $1$ oriented Riemannian submanifold with boundary. Let $n$ be the outward unit vector field on $S$. Write $\imath : S \subset M$ for the inclusion and $\mu_S$ for the inherited area form on $S$ from $M$.

\begin{proposition}\label{d is normcurl}
Assume $\dim M = 3$. Let $B$ be a vector field on $M$. Set $b = B^{\flat}$ and $\omega = \imath^*b$. Then
\begin{equation*}
d\omega = g|_S(\nabla \times B|_{S},n)\mu_S.    
\end{equation*}
\end{proposition}

\begin{proof}
We have $d\omega = \imath^*db$. Relating the Hodge stars $\star$ on $M$ and $\star_S$ on $S$ (see Proposition \ref{hodgepullback} in Appendix \ref{app:normcorrect}), we get,
\begin{equation*}
\star_S \imath^*db = \imath^*(i_n \star db|_S) = \imath^*(i_{n} (\nabla \times B|_S)^{\sharp}) = \imath^*(g|_S(n,\nabla \times B|_S)) = g|_S(n,\nabla \times B|_S).
\end{equation*}
\end{proof}

\begin{proposition}
Let $N$ be a vector field on $M$ such that $N|_S = n$. Let $B$ be a vector field on $M$. Set $b = B^{\flat}$ and $\omega = \imath^*b$. Let $\delta_S$ be the codifferential on $S$. Then
\begin{equation*}
\delta_S \omega = \nabla \cdot B -g(B,N)\nabla \cdot N - g([N,B],N)~|_S.    
\end{equation*}
\end{proposition}

\begin{proof}
We may write
\begin{equation*}
\delta_S \omega = \star_S \mathcal{L}_{\omega^{\sharp}}\mu_S.   
\end{equation*}
We also have the $\imath$-relatedness,
\begin{gather*}
T\imath \circ \omega^{\sharp} = (B-g(B,N)N) \circ \imath\\
\mu_S = \imath^*(i_N\mu).
\end{gather*}
Hence,
\begin{equation*}
\mathcal{L}_{\omega^{\sharp}}\mu_S = \imath^* \mathcal{L}_{B-g(B,N)N}i_N\mu.
\end{equation*}
With this,
\begin{equation*}
\begin{split}
\mathcal{L}_{B-g(B,N)N}i_N\mu &= \mathcal{L}_{B}i_N\mu - \mathcal{L}_{g(B,N)N}i_N\mu\\
&= ([\mathcal{L}_{B},i_N]+i_N\mathcal{L}_{B})\mu -(i_{g(B,N)N}di_X\mu + di_{g(B,N)N}i_N\mu)\\
&= (i_{[B,N]}+i_N\mathcal{L}_{B})\mu -g(B,N)i_{N}di_N\mu\\
&= i_{[B,N]}\mu+i_N\mathcal{L}_{B}\mu -g(B,N)i_{N}(i_Nd\mu + di_N\mu)\\
&= i_{[B,N]}\mu+i_N\mathcal{L}_{B}\mu -g(B,N)i_{N}\mathcal{L}_X\mu\\
&= i_{[B,N]}\mu+i_N(\nabla \cdot B)\mu -g(B,N)i_{N}(\nabla \cdot N)\mu\\
&= i_{[B,N]}\mu+(\nabla \cdot B -g(B,N)\nabla \cdot N)i_{N}\mu.
\end{split}
\end{equation*}
Now, for any vector field $Y$ on $M$, since $Y-g(X, Y) X$ is $\iota$-related to a vector field $\tilde{Y}$ on $S$ and $\imath^*\mu \in \Omega^n(S) = \{0\}$, we get, 
\begin{equation*}
\begin{split}
\imath^*(i_Y\mu) &= \imath^*(i_{Y-g(Y,N)N +g(Y,N)N}\mu)\\
&= \imath^*(i_{Y-g(Y, N)N}\mu) + g(Y,N)|_S\imath^*(i_N\mu)\\
&= i_{\tilde{Y}}\imath^*\mu + g(Y, N)|_S\mu_S\\
&= g(Y, N)|_S\mu_S. 
\end{split}
\end{equation*}
In particular, $\imath^*i_{[B,X]}\mu = g([B,X],X)|_S\mu_S$. Thus,
\begin{equation*}
\begin{split}
\mathcal{L}_{\omega^{\sharp}}\mu_S &= \imath^* \mathcal{L}_{B-g(B,N)N}i_N\mu\\
&= \imath^*(i_{[B,N]}\mu+(\nabla \cdot B -g(B,N)\nabla \cdot N)i_{N}\mu)\\
&= g([B,N],N)|_S\mu_S+(\nabla \cdot B -g(B,N)\nabla \cdot N)|_S \mu_S\\
&= (g([B,N],N)+\nabla \cdot B -g(B,N)\nabla \cdot N)|_S \mu_S.
\end{split}
\end{equation*}
Hence, we arrive at
\begin{equation*}
\begin{split}
\delta_S \omega &= \star_S \mathcal{L}_{\omega^{\sharp}}\mu_S\\
&= g([B,N],N)+\nabla \cdot B -g(B,N)\nabla \cdot N~|_S\\
&= \nabla \cdot B -g(B,N)\nabla \cdot N - g([N,B],N)~|_S.
\end{split}
\end{equation*}
\end{proof}

If $B$ is tangent to $S$, the final term $g([N,B],N)|_S$ is the normal surface derivative along $S$ in Definition \ref{normalderivative}. Moreover, since the operations are local, we may drop the assumption that $S$ has a vector field $N$ on $M$ with $N|_s = n$ since this is always true locally (see Appendix \ref{app:normcorrect}). Thus, in the tangential case, we can rewrite our formula intrinsically in terms of $B$ and $S$ a follows.

\begin{corollary}\label{cohomology is area pres}
Assume that $B$ is tangent to $S$. Set $b = B^{\flat}$ and $\omega = \imath^*b$. Let $\imath^*B$ denote the vector field on $S$ $\imath$-related to $B$. Then we get,
\begin{equation*}
\delta_S \omega = \nabla \cdot B|_S - \partial_n B.    
\end{equation*}
Moreover, if $\nabla \cdot B|_S = 0$, then if $P,u \in C^{\infty}(S)$ are with $P = e^{-u}$, then,
\begin{equation*}
\delta_S P \omega = 0 \Leftrightarrow du(\imath^*B) = \partial_n B.
\end{equation*}
\end{corollary}

\subsection{$P$-harmonic 1-forms and vector fields on an oriented Riemannian 2-torus}\label{subsec:P-harmonic}

Here, we prove theorems \ref{P-harmonic 1-forms} and \ref{P-harmonic vector fields}. Before specialising to tori, we will consider a general closed oriented Riemannian manifold $S$. Let $0 < P \in C^{\infty}(S)$. Consider the space of $P$-harmonic $k$-forms,
\begin{equation*}
\mathcal{H}^k_P(S) = \{\omega \in \Omega^k(S) : d\omega = 0,~\delta P \omega = 0\}.
\end{equation*}
It is natural to ``symmetrise" these equations by the linear automorphism on $\Omega^1(S)$ given by multiplication $\omega \mapsto P^{-1/2} \omega$, giving the isomorphism,
\begin{equation*}
\mathcal{H}^1_P(S) \cong \{\omega \in \Omega^1(S) : d_{p} \omega = 0,~\delta_{p} \omega = 0\},
\end{equation*}
where $p = \ln \sqrt{P}$ and, following Witten~\cite{witten1982supersymmetry} in his 1982 paper on supersymmetry and Morse theory, for $h \in C^{\infty}(S)$, we set
\begin{align*}
d_h  = e^{h}de^{-h} : \Omega^k(S) \to \Omega^{k+1}(S)\\
\delta_h  = e^{-h}de^{h} : \Omega^k(S) \to \Omega^{k-1}(S).
\end{align*}

Although in~\cite{witten1982supersymmetry}, the case of primary interest is $h$ being a non-degenerate Morse function, it is observed for general $h$ that the operators $d_h$, $\delta_h$ are adjoints of one another with respect to the $L^2$ inner product on $\Omega^k(M)$, $d_h^2 = 0$, $\delta_h^2 = 0$ and the relation $d_h e^{-h} = e^h d$ gives an isomorphism of the $k^{\text{th}}$ de Rham cohomology group and the cohomology group
\begin{equation*}
H^k_{\text{dR}}(S) \cong H^k_{\text{W-dR}}(S,h) = \ker(d_h : \Omega^k(S) \to \Omega^{k+1}(S))/\text{im}(d_h : \Omega^{k-1}(S) \to \Omega^{k}(S)).
\end{equation*}
The cohomology group $H^k_{\text{W-dR}}(S,h)$ is now known in the literature as a Witten deformation of the cohomology group $H^k_{\text{dR}}(S)$. In addition, Witten~\cite{witten1982supersymmetry} notes that by standard arguments, the kernel of the associated Laplace operator,
\begin{equation*}
\Delta_h = d_h\delta_h + \delta_h d_h : \Omega^k(S) \to \Omega^k(S)
\end{equation*}
has the same dimension as $H^k_{\text{dR}}(S)$, that is, the $k^{\text{th}}$ Betti number $B^k$ of $S$. Standard arguments include those present in the proof of the standard Hodge decomposition Theorem for the de Rham complex. These arguments may be adapted to a very large class of complexes known as elliptic complexes~\cite{wells1980differential} giving rise to a generalised Hodge decomposition theorem. Although, in our current position, we may already establish the following.

\begin{proposition}\label{naturalclassiso}
The assignment $\omega \mapsto [\omega]$ of a closed $k$-form on $S$ to its de Rham cohomology class gives an isomorphism $\mathcal{H}^k_P(S) \cong H^k_{\text{dR}}(S)$.
\end{proposition}

\begin{proof}
We first consider the Witten deformation to the cohomology with some $h\in C^{\infty}(S)$. From the adjointness of $d_h$ and $\delta_h$, we have at once that every $d_h$-exact $k$-form $\omega$ which is $\delta_h$-co-closed, is necessarily $\omega = 0$. Hence, the restricted linear quotient map
\begin{equation*}
\ker d_h \cap \ker \delta_h \ni \omega \mapsto [\omega]_{\text{W}} \in H^k_{\text{W-dR}}(S,h),
\end{equation*}
is an injection. Again by adjointness, $\ker d_h \cap \ker \delta_h = \ker \Delta_h$. Hence, $\dim \ker d_h \cap \ker \delta_h = B^k = \dim H^k_{\text{W-dR}}(S,h)$. So that, by the rank-nullity theorem, this linear map is an isomorphism. Setting $h = p$ and recalling the isomorphism $\omega \mapsto P^{-1/2}\omega$, the result follows.
\end{proof}

We will now draw our attention to tori. We record Proposition \ref{naturalclassiso} in this instance.

\begin{corollary}\label{1st 2 of P-harmonic}
Let $S$ be an oriented Riemannian 2-torus. Then the assignment $\omega \mapsto [\omega]$ of a closed 1-form on $S$ to its de Rham cohomology class gives an isomorphism $\mathcal{H}^1_P(S) \cong H^1_{\text{dR}}(S)$. In particular, $\dim \mathcal{H}^1_P(S) = 2$.
\end{corollary}

We will now address the non-vanishing property of the elements of $\mathcal{H}^1_P(S)$. For this, we will need to recall elements of Riemann surface theory and in particular some of its extensions due to Bers in the 1950s. Our strategy is to use a conformal equivalence to show that elements of $\mathcal{H}^1_P(S)$ give pseudo-analytic functions on a flat 2-torus; which by Bers' work on generalising the Riemann-Roch theorem, have the non-vanishing property.

\subsubsection{Elements of Riemann surface and pseudo-analytic function theory for the 2-torus}

Here we will introduce the required results from Riemann surface theory and pseudo-analytic function theory. In the next section, we will apply these results and conclude with a proof of Theorem \ref{P-harmonic 1-forms}.

Recall that a \emph{Riemann surface}, $S$, is a connected 2-manifold with a holomorphic atlas $\mathcal{A} = \{(U_{\alpha},z_{\alpha})\}$; namely, the transition maps $z_{\alpha} \circ z_{\beta}^{-1}$ of $\mathcal{A}$ are holomorphic between the open subsets $z_{\beta}(U_{\alpha} \cap U_{\beta})$ and $z_{\alpha}(U_{\alpha} \cap U_{\beta})$ of the complex plane (i.e. of $\mathbb{R}^2$). Charts $(U,z)$ on $S$ with holomorphic transition maps $z \circ z_{\beta}^{-1}$ for $(U_{\beta},z_\beta) \in \mathcal{A}$ will be called \emph{holomorphic charts}.

Following~\citet{bers1953partial}, if $F,G : S \to \mathbb{C}$ are functions such that $\text{Im}(\overline{F}G)>0$, $(F,G)$ is called a \emph{generating pair} if at every point $p \in S$, for any holomorphic chart $(U,z)$ on $S$ about $p$, the representatives of $F,G$ in the $(U,z)$ are H{\"o}lder continuous. The pair $(F,G)$ also defines a second pair $(F^*,G^*)$ given by
\begin{align*}
F^* &= \frac{2\overline{G}}{F\overline{G}-\overline{F}G},&
G^* &= \frac{2\overline{F}}{F\overline{G}-\overline{F}G}.
\end{align*}

The pair $(F,G)$ play similar roles to that of $1$ and $i$ play in the classical theory of holomorphic functions. Bers was able to conclude many similarities between holomorphic and pseudo-analytic functions; including definite orders of poles of meromorphic pseudo-analytic functions and differentials and a generalised Riemann-Roch theorem~\cite{bers1953partial}. We are interested in Bers' conclusion about differentials on a Riemann surface. We will now define these.

Bers~\cite{bers1953theory} defines a \emph{differential}, $W$, on a domain $D \subset S$ (connected and open subset of $S$) to be an assignment $(U,z) \mapsto W/dz$ where $(U,z)$ is a holomorphic chart with $U\subset D$ and $W/dz : U \to \mathbb{C}$ is a function such that, given two holomorphic charts $(U,z)$, $(V,\tilde{z})$ with $U,V \subset D$, for $p \in U \cap V$, we have
\begin{equation*}
\frac{W}{dz}\bigg|_p = \frac{W}{d\tilde{z}} \bigg|_p \frac{d\tilde{z}}{dz}\bigg|_p
\end{equation*}
where $\frac{d\tilde{z}}{dz}|_p = (\tilde{z} \circ z^{-1})'(z(p))$ denotes the complex derivative of the holomorphic transition map $\tilde{z} \circ z^{-1}$. Setting $f = \frac{W}{dz}$ and $\tilde{f} = \frac{W}{d\tilde{z}}$ this relation is also denoted $fdz = \tilde{f}d\tilde{z}$. One may also multiply $W$ by a function $f : S \to \mathbb{C}$ obtaining a differential $fW$ defined in the obvious way.

The differential $W$ is said to be \emph{continuous} (or have partial derivatives etc.) if $W/dz$ is continuous for every holomorphic chart $(U,z)$ with $U \subset D$. Similarly, the integral $\int_{\Gamma} W$ along a continuously differentiable curve $C : [0,1] \to D \subset S$ may be defined. That is, $\int_{\Gamma} W = \int_{0}^1 k(t)dt$ where $k : [0,1] \to \mathbb{C}$ is a function such that, for $t \in [0,1]$, if $(U,z)$ is a holomorphic chart with $U \subset D$, where $C(t) \in U$, then $(\varphi \circ C)'(t) \in \mathbb{C}$ is defined and we set
\begin{equation*}
k(t) = \frac{W}{dz}\bigg|_{C(t)} (\varphi \circ C)'(t).
\end{equation*}

With this, given a continuous differential $W$ on $S$ and a closed continuously differentiable curve $C : [0,1] \to S$, the number
\begin{equation*}
\text{Re}\int_C F^*W - i \text{Re}\int_C G^*W
\end{equation*}
is called the \emph{$(F,G)$-period of $W$ over $C$}. If the $(F,G)$-period vanishes over any $C$ homologous to zero, then $W$ is called a \emph{regular $(F,G)$-differential}. The following is a direct consequence of the generalised Riemann-Roch Theorem~\cite[Page 163-164]{bers1953partial}.

\begin{theorem}\label{Berstorus}
A regular $(F,G)$-differential $W$ on a closed Riemann surface $S$ of genus $g = 1$ is either identically zero or non-vanishing. 
\end{theorem}

We will now apply this result in the case of $P$-harmonic forms on an oriented Riemannian 2-torus to prove Theorem \ref{P-harmonic 1-forms}.

\subsubsection{$P$-harmonic 1-forms}

The following~\cite[Chapter 10]{taylor1996partial} is a well-known means of giving any smooth surface a Riemann surface structure.

\begin{proposition}\label{RiemanniantoRiemann}
Let $S$ be an oriented Riemannian 2-manifold. Then, there exists a maximal holomorphic atlas $\mathcal{A}$ compatible with the smooth structure on $S$. The component functions $x,y$ of holomorphic charts satisfy $\star dx = dy$ where $\star$ is the Hodge star of $S$ in $U$.
\end{proposition}

We now will now prepare an application of Bers' differentials in the context of $P$-harmonic 1-forms. On any smooth manifold $M$, the \emph{complexified cotangent bundle} $T_{\mathbb{C}}^*M$ is given by
\begin{equation*}
T_{\mathbb{C}}^*M = \sqcup_{p \in M} T_p^*M \otimes_{\mathbb{R}} \mathbb{C},
\end{equation*}
where $\mathbb{C}$ is regarded with its vector space structure over $\mathbb{R}$. Regularity and operations on sections of $T_{\mathbb{C}}^*M$ are defined component-wise. For example, for $f \in C^{\infty}(M;\mathbb{C})$, the set of smooth functions $M \to \mathbb{C}$, writing $f = u+iv$ for some unique $u,v \in C^{\infty}(M)$, we set
\begin{equation*}
df = du+idv,
\end{equation*}
where $du+idv$ is a smooth section of the complexified cotangent bundle $T_{\mathbb{C}}^*M$. If $M$ is a Riemann surface, a section $\omega$ of $T_{\mathbb{C}}^*M$ is called of \emph{type $(1,0)$} if in any holomorphic chart $(U,z)$, $\imath^*\omega = fdz$ for some function $f : U \to \mathbb{C}$, where $\imath : U \subset M$ is the inclusion. The assignment of a type $(1,0)$ section of $T_{\mathbb{C}}^*M$ to the function $f : U \to \mathbb{C}$ where $\imath^*\omega = fdz$ for any holomorphic chart $(U,z)$, is a bijective equivalence between type $(1,0)$ sections of $T_{\mathbb{C}}^*M$ and Bers' differentials on $M$. If $\omega$ is a type $(1,0)$ section of $T_{\mathbb{C}}^*M$ and $W$ is its differential equivalent, then $\omega$ is continuous if and only if $W$ is. Moreover, for any continuously differentiable curve $C : [0,1] \to M$, we have
\begin{equation*}
\int_C W = \int_C \omega,
\end{equation*}
where the latter integral is a $\mathbb{C}$-linear extension of integrals of continuous 1-forms along the curve $C$. With this, we are ready to prove Theorem \ref{P-harmonic 1-forms}.

\begin{proof}[Proof of Theorem \ref{P-harmonic 1-forms}]
Let $S$ be an oriented Riemannian 2-torus and $0<P \in C^{\infty}(S)$. Corollary \ref{1st 2 of P-harmonic} proves the first two statements of the Theorem. For the third, let $\omega$ be a $P$-harmonic 1-form. So,
\begin{align*}
d\omega &= 0,&
\delta P\omega &= 0.    
\end{align*}
Then, as previously discussed, considering the 1-form $\tilde{\omega} = P^{-1/2}\omega$, with $p = \ln \sqrt{P}$, we have that
\begin{align*}
d p \tilde{\omega} &= 0,&
d\star p^{-1} \tilde{\omega} &= 0.   
\end{align*}
Now, give $S$ the structure of a closed Riemann surface of genus $g = 1$ as per Proposition \ref{RiemanniantoRiemann}. Since $p$ is smooth, the pair $(F,G)$ given by
\begin{align*}
F &= p,&
G &= i/p,
\end{align*}
is a generating pair on $S$. Then, the pair $(F^*,G^*)$ are given by
\begin{align*}
F^* &= 1/p,&
G^* &= ip. 
\end{align*}
Consider the section $\hat{\omega}$ of $T_{\mathbb{C}}^*S$ given by
\begin{equation*}
\hat{\omega} =  \tilde{\omega} + i \star \tilde{\omega}.
\end{equation*}
Then, from $\star \hat{\omega} = i\hat{\omega}$, it easily follows from Proposition \ref{RiemanniantoRiemann} that $\hat{\omega}$ is a $T_{\mathbb{C}}^*S$ section of type $(1,0)$. Hence, we may consider its Bers' differential equivalent $W$. Let $C : [0,1] \to S$ be a continuously differentiable closed curve. Then, we get
\begin{align*}
\text{Re}\int_C {F^*~}W - i \text{Re}\int_C {G^*~} W &= \text{Re}\int_C {F^*~}\hat{\omega} - i \text{Re}\int_C {G^*~} \hat{\omega}\\
&= \int_C p^{-1}\tilde{\omega}+i\int_C p \star \tilde{\omega}.
\end{align*}
If $C$ is homologous to zero, then by Stokes' Theorem for chains, since $dp^{-1}\tilde{\omega} = 0$ and $d p \star \tilde{\omega} = 0$, we obtain that
\begin{equation*}
\text{Re}\int_C {F^*~}W - i \text{Re}\int_C {G^*~} W = 0.    
\end{equation*}
Thus, $W$ is a regular $(F,G)$ differential. Hence, from Theorem \ref{Berstorus}, we obtain that $W$ is either identically zero or non-vanishing. Thus, $\omega$ is either identically zero or non-vanishing.
\end{proof}

We will now move on to $P$-harmonic vector fields.

\subsubsection{$P$-harmonic vector fields}\label{subsec:P-harmonic vect}

We will now prove Theorem \ref{P-harmonic vector fields}, which primarily involves reinterpreting 1-form data in terms of the corresponding vector field data.

\begin{proof}[Proof of Theorem \ref{P-harmonic vector fields}]
Let $S$ be an oriented Riemannian 2-torus and $0 < P \in C^{\infty}(S)$. Denote by $\mathcal{H}_P V(S)$ the real vector space of $P$-harmonic vector fields on $S$.

\begin{proof}[Proof of statements 1 and 2]
Consider the map $\mathcal{H}_P V(S) \to \mathcal{H}_P^1(S)$ given by $X \mapsto X^{\flat}$. Then, from Theorem \ref{P-harmonic 1-forms}, we immediately obtain that $\dim \mathcal{H}_P V(S) = 2$ and that every $X \in \mathcal{H}_P V(S)$ is either identically zero or non-vanishing.
\end{proof}

For the remaining statements, let $X \in \mathcal{H}_P V(S)$ be non-vanishing and let $\gamma_1,\gamma_2$ be classes generating the first homology of $S$.

\begin{proof}[Proof of statement 3]
Set $Y = X^{\perp}/P\|X\|^2$, where $X^{\perp} = (\star X^{\flat})^{\sharp}$ is the perpendicular to $X$. The 1-forms $\omega = X^{\flat}$ and $\eta = {P \star \omega}$, are closed 1-forms on $S$ and
\begin{align*}
\omega(X/\|X\|^2) &= 1 & \eta(X/\|X\|^2) &= 0\\
\omega(Y) &= 0 & \eta(Y) &= 1.
\end{align*}
Hence, $[X/\|X\|^2,Y] = 0$. Hence, from a well-known result in Lie group theory (see Appendix \ref{app:rottranslinear}, Proposition \ref{Arnoldcommute}), $X/\|X\|^2$ is linearisable. From, for example, $\eta(X) = 0$, we see that $X$ is winding and in turn, from a well-known result in celestial mechanics (see Appendix \ref{app:rottranslinear}, Theorem \ref{celestialmechanics}) if $X$ has Diophantine winding number, then $X$ is linearisable.
\end{proof}

\begin{proof}[Proof of statement 4]
From the previous, we have a non-vanishing $\eta = {P \star \omega}$ with $\eta(X) = 0$. In particular (see Proposition \ref{correctnessrotationaltransform} in Appendix \ref{app:rottrans}), $[\eta] \neq 0$ and the winding number of $X$ with respect to $\gamma_1,\gamma_2$ is given by $[(a,b)]$ where,
\begin{align*}
a &= -\int_{\gamma_2} \eta = -\int_{\gamma_2} P \star X^{\flat},&
b &= \int_{\gamma_1} \eta = \int_{\gamma_1} P \star X^{\flat}.
\end{align*}
Moreover, for a curve $C : [0,1] \to S$ we have
\begin{equation*}
\gamma^*(P \star X^{\flat}) = P(\gamma(t)) (\star X^{\flat})(\gamma'(t))dt = P(\gamma(t)) (i_X\mu)(\gamma'(t))dt = \mu(X(\gamma(t)),\gamma'(t))dt.
\end{equation*}
So, if $\gamma_1,\gamma_2$ are represented by curves $C_1$ and $C_2$,
\begin{gather*}
a = -\int_{\gamma_2} P \star X^{\flat} = -\int_{C_2} P \star X^{\flat} = \int_0^1 \mu(X(\gamma(t)),\gamma'(t))dt,\\
b = \int_{\gamma_1} P \star X^{\flat} = \int_{C_1} P \star X^{\flat} = \int_0^1 \mu(X(\gamma(t)),\gamma'(t))dt.
\end{gather*}
\end{proof}

\begin{proof}[Proof of statement 5]
We have the linear isomorphisms
\begin{align*}
\mathcal{H}_PV(S) &\to \mathcal{H}_P^1(S), & X &\mapsto X^{\flat},\\
\mathcal{H}_P^1(S) &\to \mathcal{H}_{1/P}^1(S), & \omega &\mapsto {P\star\omega},\\
\mathcal{H}_{1/P}^1(S) &\to H^1_{\text{dR}}(S), & \omega &\mapsto [\omega].
\end{align*}
The final linear isomorphism is due to Theorem \ref{P-harmonic 1-forms}. Hence, we have the linear isomorphism
\begin{align*}
\mathcal{H}_PV(S) &\to H^1_{\text{dR}}(S),& 
X &\mapsto [{P \star X^{\flat}}].
\end{align*}
\end{proof}
Hence, by de Rham's Theorem, we have the linear isomorphism
\begin{gather*}
L : \mathcal{H}_PV(S) \to \mathbb{R}^2,\\
L(X) = \left( -\int_{\gamma_2}P\star X^{\flat},\int{\gamma_1} P \star X^{\flat} \right).
\end{gather*}
Set $\mathbb{H}_PV(S) = (\mathcal{H}_PV(S) \backslash\{0\}) / \sim$ to be the projective quotient space where $X,Y \in \mathcal{H}_PV(S)$ are equivalent $X \sim Y$ if and only if $X$ and $Y$ are linearly independent. Then, by injectivity, $L$ descends to an injection, $\mathbb{L} : \mathbb{H}_PV(S) \to \mathbb{P}(\mathbb{R})$. This map is also surjective since $L$ is. Lastly, for $X \in \mathcal{H}_PV(S)$, from statement 4, we see that $L([X])$ gives the winding number of $X$ with respect to $\gamma_1$ and $\gamma_2$. The result now follows from the definition of $\mathbb{H}_PV(S)$.
\end{proof}

We will conclude this section with a proof of Theorem \ref{mainresult}.

\begin{proof}[Proof of Theorem \ref{mainresult}]
Let $M$ be an oriented Riemannian $3$-manifold with boundary. Let $S$ be an embedded 2-torus in $M$. Let $n : S \to TM$ be a unit normal for $S$. Let $B$ be a vector field on $M$ which satisfies
\begin{align*}
\nabla \cdot B|_S &= 0,&
\nabla \times B|_S \cdot n &= 0,&
B|_S \cdot n &= 0.
\end{align*}
Consider the $\iota$-related vector field $\imath^*B$ on $S$ where $i : S \subset M$. Assume there exists a solution $u \in C^{\infty}(S)$ to the cohomological equation
\begin{equation*}
du(\imath^*B) = \partial_n B.
\end{equation*}
Now, let $b = B^{\flat}$ and $\omega = \imath^*b$. Then, from Proposition \ref{d is normcurl}, denonting by $\mu_S$ the area element on $S$,
\begin{equation*}
d\omega = g|_S(\nabla \times B|_S,n)\mu_S = 0.
\end{equation*}
From Corollary \ref{cohomology is area pres}, taking $P = e^u$, the codifferential $\delta_S$ on $S$ gives
\begin{equation*}
\delta_S P\omega = 0.    
\end{equation*}
Hence, $\omega$ is a $P$-harmonic 1-form on $S$. Considering the sharp $\sharp_S$ on $S$, $\omega^{\sharp_S} = \imath^*B$. Hence, $\imath^*B$ is a $P$-harmonic vector field on $S$. The conclusion now follows from Theorem \ref{P-harmonic vector fields}.
\end{proof}

\section{Applications and Examples}\label{sec:exandapp}

An application mentioned of Theorem \ref{mainresult} is Proposition \ref{for first integrals}. We will now prove this.

\begin{proof}[Proof of Proposition \ref{for first integrals}]
Let $M$ be an oriented Riemannian $3$-manifold with boundary. Let $S$ be a codimension 1 embedded submanifold in $M$. Let $B$ be a vector field on $M$ which satisfies $\nabla \cdot B|_S = 0$. Assume that $B$ has a first integral $\rho \in C^{\infty}(M)$ which is constant and regular on $S$.

Consider now the neighbourhood $V = \{\nabla \rho \neq 0\}$ of $S$ in $M$. Set $\tilde{u} = \|\nabla \rho\|$ and $N = \nabla f|_V /\tilde{u} |_V : V \to TM$. We have that $n = N|_S$ so that $B|_S \cdot n = 0$ and,
\begin{equation*}
\partial_n B = g|_S([N,B],N) = \frac{[N,B](\rho)}{\tilde{u}}~\bigg|_S.
\end{equation*}
Next, $B(\rho) = 0$ and $N(\rho) = g(N,\nabla \rho) = \tilde{u}$ so that,
\begin{equation*}
\frac{[N,B](\rho)}{\tilde{u}} = \frac{N(B(\rho))-B(N(\rho))}{\tilde{u}} = -B(\tilde{u})/\tilde{u} = B(-\ln\tilde{u}).
\end{equation*}
Hence, since $B$ is tangential, setting $u = -\ln \|\nabla \rho\||_S = -\ln\tilde{u}|_S$, we have
\begin{equation*}
\frac{[N,B](\rho)}{\tilde{u}}\bigg|_S = du(\imath^*B).
\end{equation*}
So that $du(\imath^*B) = \partial_n B$.
\end{proof}

On the topic of first integrals, we now provide a proof of Corollary \ref{GeneralArnold3D}.

\begin{proof}[Proof of Corollary \ref{GeneralArnold3D}]
Let $M$ be an oriented Riemannian $3$-manifold with boundary. Let $B$ be a vector field on $M$ which satisfies, for some function $\rho$,
\begin{align*}
\nabla \cdot B &= 0,& B \cdot \nabla \rho &= 0,& (\nabla \times B) \cdot \nabla \rho &= 0.
\end{align*}

First, observe the following. Let $S'$ be a closed connected component of a regular level set of $\rho$. For later, give $S'$ the orientation where the unit normal $n = \frac{\nabla \rho|_{S'}}{\|\nabla \rho\||_{S'}}$ is outward. Setting $u = -\ln\|\nabla \rho\||_{S'}$, and $\imath : S' \subset M$ to be the inclusion, our assumptions together with Proposition \ref{for first integrals} give,
\begin{align*}
\nabla \cdot B|_{S'} &= 0,& \nabla \times B|_{S'} \cdot n &= 0,& B|_{S'} \cdot n &= 0,& du(\imath^*B) &= \partial_n B.
\end{align*}
Hence, by Proposition \ref{d is normcurl} and Corollary \ref{cohomology is area pres}, $\imath^*B$ is a $P = \|\nabla \rho\||_{S'}$-harmonic vector field on $S'$.

In particular, fix a closed connected component $S$ of a regular level set of $\rho$. From the above, by Theorem \ref{mainresult}, if $B$ is not identically zero on $S$ and $S$ is a 2-torus, then $B$ is non-vanishing on $S$. Conversely, if $B$ is non-vanishing on $S$, then since $S$ is closed and connected and oriented, the non-vanishing vector field $\imath^*B$ on $S$ implies that $S$ is a 2-torus.

Assume $B$ is non-vanishing on $S$ and that $S$ is a 2-torus. Consider the neighborhood $V = \{\nabla \rho \neq 0\}$ of $S$ and the local vector field, $N = \frac{\nabla \rho|_V}{\|\nabla \rho\||_V}$. Recall that $S$ is compact. In particular, if $S \subset \partial M$, since $S$ is then embedded in $\partial M$, $S$ is a connected component of $\partial M$. Thus, if $S \subset \partial M = \emptyset$ (or $S \subset \partial M$), the Flowout Theorem~\cite[Theorem 9.20]{Lee} (the Boundary Flowout Theorem~\cite[Theorem 9.24]{Lee}) with $V$ and $N$ gives an $\epsilon > 0$ and a diffeomorphism $\Sigma : S \times I \to U$ where $I = (-\epsilon,\epsilon)$ ($I = [0,\epsilon)$) and $U$ is open subset in $M$ such that the projection $z : S \times I \to I$ satisfies $z = \rho \circ \Sigma + c$ for some constant $c \in \mathbb{R}$.

With this, let $z \in I$. Then, $S' = \Sigma(S \times \{z\})$ is an embedded 2-torus in $M$ for which $X_z = \imath^*B$ is a $P = \|\nabla \rho\||_{S'}$-harmonic vector field on $S'$ where $\imath : S' \subset M$ is the inclusion. Hence, following the proof of statement 3 in Theorem \ref{P-harmonic vector fields}, $[X/\|X\|^2,Y] = 0$ where $Y_z = X_z^{\perp}/P\|X_z\|^2$, $X_z^{\perp} = (\star X_z^{\flat})^{\sharp}$ being the perpendicular to $X_z$ and  $P = \|\nabla \rho\||_{S'}$. Thus, with the local vector fields
\begin{align*}
X &= \frac{B|_U}{\|B\|^2|_U},& Y &= \frac{\nabla \rho \times B|_U}{\|\nabla \rho\|^2 \|B\|^2|_U} = \frac{\nabla \rho \times B|_U}{\|\nabla \rho \times B\|^2|_U},
\end{align*}
we have that $[X,Y] = 0$ on $U$. Hence, using the diffeomorphism $\Sigma$ with Proposition \ref{Arnoldcommute3D} in Appendix \ref{app:rottranslinear}, there exists a diffeomorphism $\Phi : U \to \mathbb{R}^2/\mathbb{Z}^2 \times I$ where $I \subset \mathbb{R}$ is an interval, such that $B|_U$ is non-vanishing and,
\begin{equation*}
\Phi_* \left(\frac{B|_U}{\|B\|^2 |_U}\right) = a(z) \frac{\partial}{\partial x} + b(z) \frac{\partial}{\partial y}
\end{equation*}
where $a,b : I \to \mathbb{R}$ are smooth functions and $z$ is the projection onto the factor $I$.
\end{proof}

We will now clarify the application of our results to strong Beltrami fields.

\begin{proof}[Proof of Corollary \ref{advertcorollary}]
Let $M$ be a manifold with boundary embedded in $\mathbb{R}^3$ with the inherited Euclidean structure. Let $S$ be a toroidal connected component of $\partial M$. Let $n : \partial M \to \mathbb{R}^3$ be the outward unit normal on $M$. Let $B$ be a vector field on $M$ satisfying, for some $\lambda \in \mathbb{R}$,
\begin{align*}
\nabla \cdot B &= 0,&
\nabla \times B &= \lambda B,&
B|_{\partial M} \cdot n &= 0.
\end{align*}

Let $\imath : S \subset M$ denote the inclusion and consider $\imath^*B$. Corollary \ref{cohomology is area pres}, Theorem \ref{celestialmechanics} and~\cite[Lemma 2.1]{gerner2021typical} easily show that the three statements in Corollary \ref{advertcorollary} are indeed equivalent.

Lastly, consider a solution $u$ to $du(\imath^*B) = \partial_n B$ on $S$. We have that $\imath^*B$ is $P = e^{-u}$-harmonic by Proposition \ref{d is normcurl} and Corollary \ref{cohomology is area pres}. Take closed curves $C_1,C_2 : [0,1] \to S \subset \mathbb{R}^3$ whose homology classes generate the first homology $H_1(S)$ and form the integrals
\begin{equation*}
I(C_i,u) \coloneqq \int_{0}^1 \exp{(-u(C_i(t)))} \det(C_i'(t),B(C_i(t)),n(C_i(t)))dt,~i \in \{1,2 \}.
\end{equation*}
Then, by Theorem \ref{P-harmonic vector fields}, $[(-I(C_2,u),I(C_1,u))]$ in $\mathbb{P(R)}$ is the winding number and of course $(I(C_1,u),I(C_2,u))$ is Diophantine if and only if $(-I(C_2,u),I(C_1,u))$ is Diophantine (see also the proof of Proposition \ref{correctnessrotationaltransform} in Appendix \ref{app:rottranscorrect}). This gives the final part of Corollary \ref{advertcorollary}.
\end{proof}

We will now address the example on the sphere from the introduction.

\begin{proof}[Proof of Corollary \ref{divfreesphere}]
Let $B$ be a divergence-free vector field on an oriented Riemannian 3-manifold $M$ with boundary. Suppose that $B$ and $\nabla \times B$ have a mutual first integral $\rho$ on $M$ with a connected component $S$ of a regular level set diffeomorphic to $\mathbb{S}^2$. 

Just as observed in~\cite{enciso2016beltrami}, there is a neighbourhood $U$ of $S$ foliated by spheres for which $\rho$ is constant and regular thereon. Then, on such a sphere $S' \subset U$, by Proposition \ref{for first integrals} and Corollary \ref{cohomology is area pres}, setting $b = B^{\flat}$ and $\omega' = \imath'^*b$, we obtain that $\omega$ is $P$-harmonic on $S'$ with $P = \|\nabla \rho\||_S$. However, recall that from Proposition \ref{naturalclassiso}, that $\mathcal{H}^1_P(S')$ and $H^1_{\text{dR}}(S')$ are isomorphic. In particular, we have $H^1_{\text{dR}}(S') = \{0\}$ so that $\omega' = 0$. Hence, $B|_S' = 0$. Thus, $B|_U = 0$.
\end{proof}

Our results also provide some easy bounds on the size of the space of divergence-free Beltrami fields for a special class of proportionality factors in a similar vein to~\cite{enciso2016beltrami,clelland2020beltrami}.

\begin{corollary}\label{2dimBeltrami}
Let $\lambda \in C^{\infty}(M)$ be a smooth function on an oriented connected Riemannian 3-manifold with boundary which is constant and regular on $\partial M$, where $\partial M$ is diffeomorphic to a 2-torus. Then, the space of vector fields $B$ satisfying
\begin{align*}\label{divfreeBeltrami}
\nabla \cdot B &= 0,&
\nabla \times B &= \lambda B,
\end{align*}
is at most two dimensional. 
\end{corollary}

\begin{proof}
Note that our assumptions on $\lambda$ imply the vector fields under consideration are tangent to the boundary. Let $B$ be a vector field satisfying $\nabla \cdot B = 0$ and $\nabla \times B = \lambda B$. Assume that $B$ is not identically zero. Then, Vainshtein’s Lemma for abstract manifolds given by Gerner~\cite[Lemma 2.1]{gerner2021typical} immediately implies that $B$ does not entirely vanish on $\partial M$. Thus, by Proposition \ref{for first integrals} and Theorem \ref{mainresult} we get that $B$ is non-vanishing on $\partial M$. In particular, denoting by $\imath : \partial M \subset M$ the inclusion and fixing a point $p \in \partial M$, the map $B \mapsto \imath^*B|_p$ is a linear injection from the space of divergence-free Beltrami fields with proportionality factor $\lambda$ into $T_p \partial M$.
\end{proof}

We will now illustrate with a simple example that the cohomological assumption in Theorem \ref{mainresult} is not redundant.

\begin{example}
We will first consider on $\mathbb{R}^3$ the vector field
\begin{equation*}
\hat{B} = f\frac{\partial}{\partial x} + f\frac{\partial}{\partial y} + h \frac{\partial}{\partial z},
\end{equation*}
where
\begin{align*}
f &= (z+1)\cos{(x+y)},&
h &= (z^2+2z)\sin{(x+y)}.     
\end{align*}
Eventually, we will lower this vector field into the manifold $M = (\mathbb{R}/2\pi \mathbb{Z})^2 \times \mathbb{R}$ along with some of its properties. To this end, writing $\hat{S}$ for the plane $\{z=0\}$, one finds
\begin{align*}
\nabla \cdot \hat{B} &= 0,&
h |_{\hat{S}} &= 0,&
\nabla \times \hat{B} \cdot \frac{\partial}{\partial z} &= 0.
\end{align*}
and denoting by $\hat{\imath} : \hat{S} \subset \mathbb{R}^3$ the inclusion, for any $u \in C^{\infty}(\hat{S})$, we get
\begin{equation*}
du(\imath^*\hat{B}) = u_x f|_{z = 0} + u_y f|_{z = 0} = (u_x+u_y)\cos{(x+y)}.
\end{equation*}
on the other hand, since $\frac{\partial}{\partial z}$ is a geodesic vector field of unit length,
\begin{equation*}
\partial_{\hat{n}}\hat{B} = \partial_z \cdot [\hat{B},\partial_z] = \partial_z(\hat{B} \cdot \partial_z) = h_z|_{z=0} = 2\sin{(x+y)}.
\end{equation*}
Hence, $du(\imath^*\hat{B}) = \partial_{\hat{n}}\hat{B}$ has no solutions. Since $f$ and $h$ are $2\pi$-periodic in $x$ and $y$, $\hat{B}$ descends to a vector field $B$ on $M$. Given the inherited oriented Riemannian structure from $\mathbb{R}^3$ on $M$, we have
\begin{align*}
\nabla \cdot B &= 0,&
\nabla \times B|_{S} &= 0,&
\nabla \times B|_{S} \cdot n  &= 0,
\end{align*}
where $S$ is the 2-torus embedded in $M$ lowered from $\hat{S}$ in $\mathbb{R}^3$ with outward unit normal $n$. The other data shows that there does not exist a solution $u \in C^{\infty}(S)$ to $du(\imath^*\hat{B}) = \partial_{n}B$. This is expected from Theorem \ref{mainresult} since we see that $B$ has many zeros on $S$ but is not identically zero on $S$.
\end{example}

We will now illustrate how the cohomology class of a non-vanishing closed 1-form $\omega$ on a 2-torus $S$ does not explain the topology of the integral curves of $\omega^{\sharp}$. \newline

\begin{example}
Consider the non-vanishing closed 1-forms,
\begin{align*}
\hat{\omega} &= dy,&
\hat{\eta} &= d\sin{x} + dy = \cos{x}dx + dy.
\end{align*}
Then, $\hat{\omega}$ and $\hat{\eta}$ descend to non-vanishing cohomologous 1-forms $\omega$ and $\eta$ on the 2-torus $S = \mathbb{R}^2/2\pi \mathbb{Z}^2$. We see that $\omega^{\sharp}$ is periodic. However, $\eta^{\sharp}$ is not periodic. To see this, consider $X = \hat{\eta}^{\sharp}$ and let $C$ be an integral curve of $X$ and $T > 0$. Writing $C = (C^1,C^2)$, we have
\begin{equation*}
\dot{C} = (\dot{C}^1,\dot{C}^2) = (\cos C^1,1).
\end{equation*}
In particular,
\begin{equation*}
\dot{C}^1(t)^2 = \cos {(C^1(t))} \dot{C}^1(t) = \frac{d}{dt} \sin{(C^1(t))}.
\end{equation*}
Hence,
\begin{equation*}
\int_0^T \cos^2{(C^1(t))} dt = \int_0^T \dot{C}^1(t)^2 dt = \sin{(C^1(T))}-\sin{(C^1(0))}.
\end{equation*}
Now, assume $p = C(0)$ is with $\cos p^1 \neq 0$. Then, by continuity, we must have
\begin{equation*}
\int_0^T \cos^2{(C^1(t))} dt > 0.
\end{equation*}
So that
\begin{equation*}
\sin{(C^1(T))} \neq \sin{(C^1(0))}.
\end{equation*}
Hence
\begin{equation*}
\nexists k \in \mathbb{Z} \text{ such that } C^1(T) = C^1(0) + 2\pi k.
\end{equation*}
Thus, in $S$, $C(T) + 2\pi\mathbb{Z}^2 \neq  C(0) + 2\pi\mathbb{Z}^2$. Hence, $\eta^{\sharp}$ is not periodic. On the other hand, the perpendicular vector fields $W = (\star\omega)^{\sharp}$ and $H = (\star\eta)^{\sharp}$ have the same winding number with respect to every pair of generators since $[\omega] = [\eta]$ and the flows of $W$ and $H$ are explicitly given by
\begin{align*}
\psi^W(t,p + 2\pi\mathbb{Z}^2) &= p-t(1,0) + 2\pi\mathbb{Z}^2 ,&
\psi^H(t,p + 2\pi\mathbb{Z}^2) &= p-t(1,0) + (0,\sin{p^1}+\sin{(t-p^1)}) + 2\pi\mathbb{Z}^2.
\end{align*}
The curves $t\mapsto \psi^W(2\pi t,p+ 2\pi\mathbb{Z}^2)$ and $t \mapsto \psi^H(2\pi t,p+ 2\pi\mathbb{Z}^2)$, $t \in [0,1]$, are closed and easily seen to be homologous to the curve $C_1 : [0,1] \to S$ where $C_1(t) = [-t(1,0)]$. This situation holds much more generally; as will be discussed in a future paper which will further explore the rotational transform in the context of magnetic confinement fusion.
\end{example}

\section{Discussion}\label{sec:RigidityWitten}

Theorem \ref{P-harmonic 1-forms} in the case of constant $P$ may be proven with the classical Riemann-Roch Theorem. In fact, together with Calibi's work on intrinsically harmonic forms~\cite{calabi1969intrinsic}, one can establish the following known result.

\begin{proposition}\label{Calibitorus}
Let $S$ be an oriented 2-torus. Let $\omega$ be a closed 1-form. Then, the following are equivalent.
\begin{enumerate}
    \item $\delta\omega = 0$ for some metric $g$ on $S$.
    \item $\omega$ is either identically zero or non-vanishing.
\end{enumerate}
\end{proposition}

This is proven in the Appendix \ref{app:rottranslinear} and is the covariant approach we take to prove the other known result; namely Theorem \ref{celestialmechanics}, which is also proven there. It would be preferable to shorten the proof of Theorem \ref{P-harmonic 1-forms} by a suitable reduction to the constant $P$ case. For instance, from the conclusion of Theorem \ref{P-harmonic 1-forms} we get the following.

\begin{corollary}\label{noncanonicalmetrics}
Let $S$ be an oriented 2-torus and $0 < P \in C^{\infty}(S)$. Let $g$ be a metric on $S$ and consider the $P$-harmonic 1-forms $\mathcal{H}_P^1(S,g)$ on $(S,g)$. Let $\omega,\eta \in \mathcal{H}_P^1(S,g)$ form a basis for $\mathcal{H}_P^1(S,g)$. Then, there exists a metric $\tilde{g}$ with induced Hodge star $\tilde{\star}$ satisfying
\begin{equation*}
\tilde{\star} \omega = \eta
\end{equation*}
so that $\mathcal{H}_P^1(S,g) = \mathcal{H}^1(S,\tilde{g})$, the harmonic 1-forms on $(S,\tilde{g})$.
\end{corollary}

However, the metrics in Corollary \ref{noncanonicalmetrics} are in a sense retrospective of the conclusion of Theorem \ref{P-harmonic 1-forms} and are non-canonical. More precisely, the Witten-deformed co-differential $\delta_p = P^{-1}\delta P$ on $(S,G)$ has a different kernel to the co-differential $\tilde{\delta}$ on $(S,\tilde{g})$ when $P$ is non-constant in Corollary \ref{noncanonicalmetrics}. This may be shown in higher generally as a comparison between different Witten-deformed co-differentials on $S$, without reference to Theorem \ref{P-harmonic 1-forms}, as follows.

\begin{proposition}
Let $S$ be an oriented Riemannian 2-torus. Let $g$ and $\tilde{g}$ be Riemannian metrics on $S$. Let $0 < P,Q \in C^{\infty}(S)$ and consider the operators
\begin{align*}
\delta_p &= P^{-1}\delta P,&
\tilde{\delta}_q &= Q^{-1}\tilde{\delta} Q.
\end{align*}
Suppose that
\begin{equation*}
\ker \delta_p \subset \ker \tilde{\delta}_q.
\end{equation*}
Then $g$ and $\tilde{g}$ are conformally equivalent and $P = c Q$ for some constant $c > 0$.
\end{proposition}

\begin{proof}
Let $R = P^{-1}Q$, $r = \ln \sqrt {R}$ and consider $\tilde{\delta}_r = R^{-1}\tilde{\delta} R$. Then, we have $\ker \delta \subset \ker \tilde{\delta}_r$. Considering the bundle metrics $\langle, \rangle$ and $\langle, \rangle_{\sim}$ on $\Omega^1(S)$, we have for any $f \in C^{\infty}(S)$ and $\omega \in \Omega^1(S)$ that
\begin{align*}
\delta (f \omega) &= \delta f - \langle df,\omega \rangle,\\
\tilde{\delta}_r (f \omega) &= \tilde{\delta}_r \omega -\langle df,\omega \rangle_{\sim}.
\end{align*}
In particular, for any $f \in C^{\infty}(S)$ and $g$-harmonic 1-form $\zeta$, because also $\delta_w\zeta = 0$, we get
\begin{align*}
\delta (f \zeta) &= -\langle df,\zeta \rangle,\\
\tilde{\delta}_r (f \zeta) &= -\langle df,\zeta \rangle_{\sim}.
\end{align*}

Take two $g$-harmonic 1-forms $\chi^1,\chi^2$ which generate the first cohomololgy on $S$. Now, set $\zeta^i = \star\chi^i$. Then, for each point $p \in S$, $(\zeta^1|_p,\zeta^2|_p)$ forms a basis of $\Lambda^1(T_pS)$. For $(a,b) \in \mathbb{R}^2$, form $\zeta_{(a,b)} = a\zeta^1+b\zeta^2$. We claim that, for all $(a,b) \in \mathbb{R}^2$ such that $(a,b) = c q$ for some $c \in \mathbb{R}$ and $q \in \mathbb{Q}^2$, for any point $p \in S$, there exists $f \in C^{\infty}(S)$ such that $df|_p \neq 0$ and $\langle df,\zeta_{(a,b)} \rangle = 0$.

Indeed, let $0 \neq (a,b) \in \mathbb{R}^2$. Consider the dual vector field $X = \zeta_{(a,b)}^{\sharp}$ in the metric $g$. Then, we have the non-vanishing form $\omega = a\chi^1+b\chi^2$ with $\omega(X) = 0$. Hence, $X$ has winding number $[(a,b)]$ and is non-vanishing. By Theorem \ref{celestialmechanics}, and Proposition \ref{compatibilityrotationaltransform} in Appendix \ref{app:rottranscorrect}, this means there exists a diffeomorphism $\Phi : S \to \mathbb{R}^2/\mathbb{Z}^2$ and a function $0 < f \in C^{\infty}(S)$ such that
\begin{equation*}
\Phi_*(X/f) = a\frac{\partial}{\partial x} + b\frac{\partial}{\partial y}.
\end{equation*}
Now, assume that $(a,b) = c q$ for some $c \in \mathbb{R}$ and $q \in \mathbb{Q}^2$. Then, there exists $0 \neq (m,n) \in \mathbb{Z}^2$ such that $(m,n) \cdot (a,b) = 0$. Consider $h_1,h_2 \in C^{\infty}(\mathbb{R}^2)$ given by
\begin{align*}
h_1 &= \cos(2\pi(mx+ny)),&
h_2 &= \sin(2\pi(mx+ny)). 
\end{align*}
We have that, $h_1,h_2$ are invariant under $\mathbb{Z}^2$-translations. Hence, $h_1,h_2$ descend to $h_1',h_2' \in C^{\infty}(\mathbb{R}^2/\mathbb{Z}^2)$. Then, considering $f_1 = h_1' \circ \Phi, g_2 = h_2' \circ \Phi \in C^{\infty}(S)$, we see that $df_1(X) = 0 = df_2(X)$ and that for any point $p \in S$, either $df_1|_p \neq 0$ or $df_2|_p \neq 0$. This proves the claim.

With this, let $p \in S$. Then, let $(a,b) \in \mathbb{R}^2$ such that $(a,b) = c q$ for some $c \in \mathbb{R}$ and $q \in \mathbb{Q}^2$. Then, consider the form $\eta = a\zeta^1|_p + b\zeta^2|_p \in \Lambda^1(T_pS)$. Then, take $f \in C^{\infty}(S)$ such that $df|_p \neq 0$ and $\langle df,\zeta_{(a,b)} \rangle = 0$. Then, we have
\begin{equation*}
\delta (f \zeta) = -\langle df,\zeta \rangle = 0.
\end{equation*}
So that
\begin{equation*}
-\langle df,\zeta \rangle_{\sim} = \tilde{\delta}_r (f \zeta) = 0.
\end{equation*}
Hence, because $df|_p \neq 0$,
\begin{equation*}
\langle \tilde{\star}\eta,\eta \rangle = 0.    
\end{equation*}
Hence, since $\mathbb{Q}^2$ is dense in $\mathbb{R}^2$, with respect to the topology induced by the inner product $\langle,\rangle$ on $\Lambda^1(T_pS)$, there exists a dense subset $S$ such that, for all $\eta \in S$, $\langle \tilde{\star}\eta,\eta \rangle = 0$. Thus, since $\tilde{\star}$ is a linear operator, $\langle \tilde{\star}\eta,\eta \rangle = 0$ for all $\eta \in \Lambda^1(S)$. Now, considering $\star$ on $\Lambda^1(T_pS)$, we see that, for all $\eta \in \Lambda^1(T_pS)$, there exists $c \in \mathbb{R}$ such that $\tilde{\star} \eta = c \star \eta$. Since $\tilde{\star}$ and $\star$ are linear operators and square to $-\text{Id}$, it follows that for $\eta \in \Lambda^1(S)$, $\tilde{\star} \eta = \star \eta$. In particular, $g|_p$ and $\tilde{g}|_p$ are conformally equivalent. Since $p \in S$ was arbitrary, $g$ and $\tilde{g}$ are conformally equivalent.

On top-forms, $\tilde{\star} = \kappa \star$ and as in the above, on 1-forms, $\tilde{\star} = \star$. Hence
\begin{equation*}
\tilde{\delta}_r = R^{-1}\tilde{\delta} R = \kappa \delta_r.      
\end{equation*}
Hence, $\ker \delta \subset \ker \delta_r$. Now, we have for $\omega \in \Omega^1(S)$ that
\begin{equation*}
\delta_r \omega = \delta \omega - \langle dw,\omega\rangle.
\end{equation*}
Hence, for any $g$-harmonic 1-form $\zeta$, we have
\begin{equation*}
0 = \delta_r \zeta = \delta \omega - \langle dr,\omega\rangle = 0 - \langle dr,\zeta \rangle = \langle dr,\zeta \rangle.
\end{equation*}
Thus, $r$ is constant. Hence, $R$ is constant. Thus, $P = cQ$ for some constant $c > 0$.
\end{proof}

In this way, it is seen that $P$-harmonic 1-forms on surfaces have their differences with harmonic 1-forms. On the other hand, they have many similarities: as seen in Section \ref{subsec:P-harmonic}, Theorem \ref{P-harmonic 1-forms}, and Corollary \ref{noncanonicalmetrics}.

\section{Competing interests}

The authors declare no competing interests.

\section{Acknowledgements}

This paper was written while the first author received an Australian Government Research Training Program Scholarship at The University of Western Australia.

\appendix

\section{The normal surface derivative and a Hodge-star formula}\label{app:normcorrect}

The following addresses the correctness of Definition \ref{normalderivative}.

\begin{proposition}[Correctness of Definition \ref{normalderivative}]
Let $S$ be an oriented codimension 1 Riemannian embedded submanifold with boundary of a manifold with boundary $M$. Then the following holds.
\begin{enumerate}
    \item Let $\mathcal{V} : S \to TM$ be a non-vanishing normal-pointing vector field on $S$. Then for any point $p \in S$, there exists a local vector field $V : U \to TM$ extending $\mathcal{V}|_{U \cap S}$.
    \item Let $B$ be a vector field which is tangent to $S$. Let $p \in S$ and $V_{i} : U_i \to TM$ for $i \in \{1,2\}$ be local vector fields on neighbourhoods $U_i$ of $p$ such that $V_i|_{U_i \cap S} = \mathcal{V}_{U_j \cap S}$. Then $g|_p([V_1,B]|_p,V_1|_p) = g|_p([V_2,B]|_p,V_2|_p)$.
\end{enumerate}
\end{proposition}

\begin{proof}
Set $n = \dim M$. Use the boundaryless double as in~\cite[Example 9.32]{Lee} to give a manifold $\tilde{M}$ such that $M$ is a regular domain of $\tilde{M}$. That is, a properly embedded codimension $0$ submanifold with boundary of $\tilde{M}$. Then, $S$ is an codimension $1$ embedded submanifold of $\tilde{M}$ and thus obeys a local $(n-1)$-slice condition for submanifolds with boundary~\cite[Theorem 5.51]{Lee}. That is, for each point $p \in S$, there exists a chart $(\tilde{U},\varphi)$ in $\tilde{M}$ such that either $\tilde{U} \cap S = \{\varphi^n = 0,~\varphi^{n-1} \geq 0\}$ or $\tilde{U} \cap S = \{\varphi^n = 0\}$. From this, we easily obtain the following. For any $p \in S$, there exists a neighbourhood $U$ in $M$ and a regular function $f \in C^{\infty}(U)$ which is constant on $U \cap S$ and for any $h \in C^{\infty}(S)$, there exists a $H \in C^{\infty}(U)$ with $H|_{U \cap S} = h|_{U \cap S}$.

For the first part, since $\mathcal{V}$ is non-vanishing, $h = \|\mathcal{V}\| \in C^{\infty}(S)$. Moreover, for any $p \in S$, there exists a neighbourhood $U$ in $M$ and a regular function $f \in C^{\infty}(U)$ which is constant on $U \cap S$ and a $H \in C^{\infty}(U)$ with $H|_{U \cap S} = h|_{U \cap S}$. Then, since both $\mathcal{V}$ and $\nabla f|_S$ are normal-pointing, the local vector field $V = \frac{H}{\|\nabla f\|} \nabla f$ extends $\mathcal{V}|_{U \cap S}$.

For the second part, take a neighbourhood $U \subset U_1 \cap U_2$ of $p$ in $M$ a regular function $f \in C^{\infty}(U)$ which is constant on $U \cap S$. Consider the local vector field $\nabla f/\|\nabla f\|$ on $U$. We get
\begin{equation*}
g|_p([V_i,B]|_p,V_i|_p) =  \frac{\|\mathcal{V}|_p\|}{\|\nabla f|_p\|}  g([V_i,B], \nabla f)|_p.
\end{equation*}
Moreover,
\begin{align*}
g([V_i,B], \nabla f) &= df([V_i,B])\\
&= [V_i,B](f)\\
&= V_i(B(f)) - B(V_i(f))\\
&= d(B(f))(V_i) - d(V_i(f))(B).
\end{align*}
Now, $d(B(f))(V_i)|_p = d(B(f))|_p(\mathcal{V}|_p)$ and denoting by $\imath : S \subset M$ we have
\begin{align*}
d(V_i(f))(B)|_p &= d(V_i(f)|_{S \cap U})|_p(\imath^*B|_p)\\
&= d(df(V_i)|_{S \cap U})|_p(\imath^*B|_p)\\
&= d(df(\mathcal{V})|_{S \cap U})|_p(\imath^*B|_p).
\end{align*}
From this, we see that $g|_p([V_1,B]|_p,V_1|_p) = g|_p([V_2,B]|_p,V_2|_p)$.
\end{proof}

The following establishes the Hodge star formula used in Proposition \ref{d is normcurl}. It suffices to consider the case of vector spaces because the Hodge star is defined point-wise. We will do this in arbitrary dimensions because no additional difficulty is met.

\begin{proposition}\label{hodgepullback}
Let $(V,g,\mathcal{O})$ be an oriented inner product space of dimension $n \geq 2$. Let $S$ be codimension $1$ vector subspace with a unit normal $n$. Consider the induced oriented inner product space $(S,i^*g,\mathcal{O}_S)$ where $i : S \subset V$ is the inclusion. Consider the Hodge star $\star$ on $V$ and $\star_S$ on $S$. Then, for any $k$-form $\omega \in \Lambda^k(V)$,
\begin{equation*}
\star_S \imath^*\omega = (-1)^{k(n-k)}\imath^*(i_n \star \omega).
\end{equation*}
\end{proposition}

\begin{proof}
Let $\eta,\omega \in \Lambda^k(S)$. Consider the projection $T : V \to S$ such that, $T(v) = v-g(n,v)n$. Consider then $\tilde{\eta} = T^*\eta \in \Lambda^1(V)$. Since $T \circ i$ is the identity on $S$, $\eta = i^*\tilde{\eta}$. Now,
\begin{equation*}
\begin{split}
\eta \wedge  (-1)^{k(n-k)}\imath^*(\star i_n \omega) &= (-1)^{k(n-k)} \eta \wedge \imath^*(\star i_n \omega)\\
&= (-1)^{k(n-k)} i^*\tilde{\eta} \wedge \imath^*(i_n \star \omega)\\
&= \imath^*( (-1)^{k(n-k)} \tilde{\eta} \wedge (i_n \star \omega)).
\end{split}
\end{equation*}
Now, since $T(n) = 0$, we have $i_n\tilde{\eta} = 0$ so that
\begin{gather*}
i_n(\tilde{\eta} \wedge \star \omega) = i_n\tilde{\eta} \wedge \star \omega + (-1)^{k(n-k)}\tilde{\eta}\wedge (i_n\star \omega) = (-1)^{k(n-k)}\tilde{\eta}\wedge (i_n\star \omega),\\
i_n(\tilde{\eta} \wedge \star \omega) = i_n(\langle \tilde{\eta},\omega \rangle_{\Lambda^k(V)} \mu) = \langle \tilde{\eta},\omega \rangle_{\Lambda^k(V)} i_n\mu,\\
\Rightarrow (-1)^{k(n-k)}\tilde{\eta}\wedge (i_n\star \omega) = \langle \tilde{\eta},\omega \rangle_{\Lambda^k(V)} i_n\mu.
\end{gather*}

Hence,
\begin{equation*}
\begin{split}
\eta \wedge  (-1)^{k(n-k)}\imath^*(\star i_n \omega) &= \imath^*( (-1)^{k(n-k)} \tilde{\eta} \wedge (i_n \star \omega))\\
&= \imath^*(\langle \tilde{\eta},\omega \rangle_{\Lambda^k(V)} i_n\mu)\\
&= \langle \tilde{\eta},\omega \rangle_{\Lambda^k(V)} \imath^*(i_n\mu).
\end{split}
\end{equation*}

Let $(e_2,...,e_n)$ be an orthonormal basis of $S$ with the correct orientation. Then, $(e_1,e_2,...,e_n)$ is an orthonormal basis of $V$ where $e_1 = n$. Then,
\begin{equation*}
\begin{split}
\langle \tilde{\eta},\omega \rangle_{\Lambda^k(V)} &= \frac{1}{k!}\sum_{1 \leq j_1,...,j_k \leq n}\tilde{\eta}(e_{j_1},...,e_{j_k})\omega(e_{j_1},...,e_{j_k}).
\end{split}
\end{equation*}
Moreover, if $j_i = 1$ for some $i \in \{1,...,k\}$, we have $T(e_{j_i}) = 0$ and hence, $\tilde{\eta}(e_{j_1},...,e_{j_k})\omega(e_{j_1},...,e_{j_k}) = 0$. Setting $(f_1,...,f_{n-1}) = (e_2,...,e_n)$ we have
\begin{equation*}
\begin{split}
\sum_{1 \leq j_1,...,j_k \leq n}\tilde{\eta}(e_{j_1},...,e_{j_k})\omega(e_{j_1},...,e_{j_k}) &= \sum_{2 \leq j_1,...,j_k \leq n}\tilde{\eta}(e_{j_1},...,e_{j_k})\omega(e_{j_1},...,e_{j_k})\\
&= \sum_{1 \leq j_1,...,j_k \leq n-1}\tilde{\eta}(f_{j_1},...,f_{j_k})\omega(f_{j_1},...,f_{j_k}).
\end{split}    
\end{equation*}
Then, for $i \in \{1,...,n-1\}$ since $f_i \in S$, $T(f_i) = f_i$. Hence, for $1 \leq j_1,...,j_k \leq n-1$, we have
\begin{equation*}
\tilde{\eta}(f_{j_1},...,f_{j_k})\omega(f_{j_1},...,f_{j_k}) = \eta(f_{j_1},...,f_{j_k})(i^*\omega)(f_{j_1},...,f_{j_k}).    
\end{equation*}
Hence,
\begin{equation*}
\begin{split}
\langle \tilde{\eta},\omega \rangle_{\Lambda^k(V)} &= \frac{1}{k!}\sum_{1 \leq j_1,...,j_k \leq n}\tilde{\eta}(e_{j_1},...,e_{j_k})\omega(e_{j_1},...,e_{j_k})\\
&= \frac{1}{k!}\sum_{1 \leq j_1,...,j_k \leq n-1}\eta(f_{j_1},...,f_{j_k})(i^*\omega)(f_{j_1},...,f_{j_k})\\
&= \langle \eta,i^*\omega \rangle_{\Lambda^k(S)}.
\end{split}
\end{equation*}
One may easily check with the orthonormal bases introduced that $\mu_S = i^*(i_n\mu)$ so that,
\begin{equation*}
\begin{split}
\eta \wedge \imath^*(\star i_n \omega) &= i^*\tilde{\eta} \wedge \imath^*(i_n \star \omega)\\
&= \imath^*(\tilde{\eta} \wedge (i_n \star \omega))\\
&= \imath^*((-1)^{k(n-k)}\langle \tilde{\eta},\omega \rangle_{\Lambda^k(V)} i_n\mu)\\
&= (-1)^{k(n-k)}\langle \tilde{\eta},\omega \rangle_{\Lambda^k(V)} \imath^*(i_n\mu)\\
&= (-1)^{k(n-k)}\langle \eta,i^*\omega \rangle_{\Lambda^k(S)} \mu_S.
\end{split}
\end{equation*}
Thus,
\begin{equation*}
\eta \wedge (-1)^{k(n-k)}\imath^*(\star i_n \omega) = \langle \eta,i^*\omega \rangle_{\Lambda^k(S)} \mu_S. 
\end{equation*}
Hence,
\begin{equation*}
\star_S \imath^*\omega = (-1)^{k(n-k)}\imath^*(i_n \star \omega).
\end{equation*}
\end{proof}

\section{The winding number}\label{app:rottrans}

Besides proving correctness of definitions related to the winding number, we hope to highlight the elegance of cohomological and covariant approaches to basic properties of the winding number. Specifically, we will introduce the notions of intrinsically harmonic 1-forms and cohomologically rigid vector fields once they are needed.

\subsection{Correctness and compatibility of definitions}\label{app:rottranscorrect}

A large part of the correctness of Definition \ref{rotationaltransform} comes from the Poincar{\'e} Duality Theorem, which we emphasise in the proof.

\begin{proposition}[Correctness of Definition \ref{rotationaltransform}]\label{correctnessrotationaltransform}
The following holds.
\begin{enumerate}
    \item If $\omega_0$ is a non-vanishing closed 1-form, then $[\omega_0] \neq 0$.
    \item Let $X$ be a vector densely-non-vanishing vector field. Let $\gamma_1,\gamma_2 \in H_1(S)$ generate $H_1(S)$. Assume that $\omega,\eta$ are closed 1-forms with non-trivial cohomology classes $[\omega] \neq 0 \neq [\eta]$ such that $\omega(X) = 0 = \eta(X)$. Then, setting
    \begin{align*}
    \omega_i &\coloneqq \int_{\gamma_i}\omega,&
    \eta_i &\coloneqq \int_{\gamma_i}\eta,~i \in \{1,2\},
    \end{align*}
    we have that $(-\omega_2,\omega_1) \neq 0 \neq (-\eta_2,\eta_1)$ and, in $\mathbb{P(R)}$, $[(-\omega_2,\omega_1)] = [(-\eta_2,\eta_1)]$. 
    \item Let $\tilde{\gamma}_1,\tilde{\gamma}_2 \in H_1(S)$ generate $H_1(S)$ and set
    \begin{equation*}
    \tilde{\eta}_i \coloneqq \int_{\tilde{\gamma}_i}\eta,~i \in \{1,2\}.
    \end{equation*}
    Then $(-\omega_2,\omega_1)$ is a Diophantine if and only if $(-\tilde{\eta}_2,\tilde{\eta}_1)$ is a Diophantine. 
\end{enumerate}
\end{proposition}

\begin{proof}

\begin{proof}[Proof of statement 1]
If $[\omega_0] = 0$, then $\omega_0 = df$ for some $f \in C^{\infty}(S)$. At an extremising point $p \in S$ of $f$, $\omega_0|_p = df|_p = 0$. Such a point exists by compactness of $S$. Hence, we must have $[\omega_0] \neq 0$.
\end{proof}

\begin{proof}[Proof of statement 2]
First let $\alpha,\beta$ be closed 1-forms such that
\begin{equation*}
\int_S \alpha \wedge \beta = 0.
\end{equation*}
Denoting by $\mathcal{Z}^1(S)$ the set of closed 1-forms on $S$, from the Poincar{\'e} Duality Theorem, we have the isomorphism $\text{PD} : H_{\text{dR}}^1(S) \to H_{\text{dR}}^1(S)^*$, given by
\begin{equation*}
 \text{PD}([\alpha'])([\beta']) = \int_S \alpha' \wedge \beta',~\alpha,\beta \in \mathcal{Z}^1(S).
\end{equation*}
Then, since $\text{PD}([\alpha])(\lambda [\alpha]) = 0$ for all $\lambda \in \mathbb{R}$, we get by The Rank-Nullity Theorem and the fact that $\dim H^1_{\text{dR}}(S) = 2$, that, $[\alpha]$ and $[\beta]$ must be linearly dependent. With this, letting $\mu \in \Omega^2(S)$ be a volume element and $f \in C^{\infty}(S)$ such that $\omega \wedge \eta = f\mu$, we have
\begin{gather*}
i_X(\omega \wedge \eta) = \omega(X)\eta - \eta(X)\omega = 0,\\
\Rightarrow f i_X\mu = i_Xf\mu = 0.
\end{gather*}
Now, for all $p \in S$ with $X|_p \neq 0$, we have $i_X\mu|_p \neq 0$ so that $f|_p = 0$. Thus, $f$ is densely-vanishing so that, since $f \in C^{\infty}(S)$, $f = 0$. Thus,
\begin{equation*}
\int_{S}\omega \wedge \eta = \int_{S} 0 = 0.
\end{equation*}
Hence, by the above, $[\omega]$ and $[\eta]$ are linearly dependent. Hence, the vectors $(-\omega_2,\omega_1), (-\eta_2,\eta_1)$ are linearly dependent in $\mathbb{R}^2$. Moreover, since $[\omega] \neq 0 \neq [\eta]$ we have by de Rham's Theorem that $(-\omega_2,\omega_1) \neq 0 \neq  (-\eta_2,\eta_1)$. Thus, in $\mathbb{P(R)}$, $[(-\omega_2,\omega_1)] = [(-\eta_2,\eta_1)]$. 
\end{proof}

\begin{proof}[Proof of statement 3]
Consider first the vectors
\begin{equation*}
v = (v_1,v_2) = \left(\int_{\gamma_1}\omega,\int_{\gamma_2}\omega \right),~ \tilde{v} = (\tilde{v}_1,\tilde{v}_2) = \left(\int_{\tilde{\gamma}_1}\omega,\int_{\tilde{\gamma}_2}\omega \right).
\end{equation*}
Now, there exist integer matrices $A,\tilde{A} \in \mathbb{Z}^{2\times 2}$ such that, $\tilde{\gamma}_i =  A_{ij}\gamma_j$ and $\gamma_i = \tilde{A}_{ij}\tilde{\gamma}_j$. From linear independence, we see that $A\tilde{A} = \tilde{A}A = I$, the identity $2\times 2$ matrix. Hence, $A \in GL(2,\mathbb{Z})$. Hence,
\begin{equation*}
\tilde{v}_i = \int_{\tilde{\gamma}_i}\omega = \int_{{A}_{ij}\gamma_j}\omega = {A}_{ij}\int_{\gamma_j}\omega = A_{ij}v_{j}.
\end{equation*}
The $2\times 2$ matrix $R$ with $R(x,y) = (-y,x)$ for $(x,y) \in \mathbb{R}^2$ is in $GL(2,\mathbb{Z})$. Moreover,
\begin{equation*}
R A R (-\omega_2,\omega_1) = R A v = R\tilde{v} = (-\tilde{\omega}_2,\tilde{\omega}_1)
\end{equation*}
and $RAR \in GL(2,\mathbb{Z})$. Now, in general, let $L \in GL(2,\mathbb{Z})$ and $u \in \mathbb{R}^2$ be a Diophantine vector. So that, there exists $\gamma > 0$ and $\tau > 1$ such that, for all $k \in \mathbb{Z}^2\backslash \{0\}$,
\begin{equation*}
|\langle u,k \rangle| \geq \gamma \|k\|^{-\tau}.
\end{equation*}
Consider $\tilde{u} = L u$. Then, fix a constant $C > 0$ such that, $\|L^Tu'\| \leq C\|u'\|$ for all $u' \in \mathbb{R}^2$ and set $\tilde{\gamma} = \gamma C^{-\tau}$. Then, for $k \in \mathbb{Z}^2 \backslash \{0\}$ non-zero, we have $L^Tk \in \mathbb{Z}^2 \backslash \{0\}$ so that,
\begin{gather*}
|\langle \tilde{u},k \rangle| = |\langle L u,k \rangle| = |\langle u,L^T k \rangle| \geq \gamma \|L^T k\|^{-\tau} \geq \gamma (C\|k\|)^{-\tau} = \tilde{\gamma}\|k\|^{-\tau}.  
\end{gather*}
Hence, $L u = \tilde{u}$ is Diophantine. With this, we see that $(-\omega_2,\omega_1)$ is Diophantine if and only if $(-\tilde{\omega}_2,\tilde{\omega}_1)$ is Diophantine. Moreover, by statement 2, we have that $[(-\tilde{\omega}_2,\tilde{\omega}_1)] = [(-\tilde{\eta}_2,\tilde{\eta}_1)]$ so that clearly $(-\tilde{\omega}_2,\tilde{\omega}_1)$ is Diophantine if and only if $(-\tilde{\eta}_2,\tilde{\eta}_1)$. In total, $(-\omega_2,\omega_1)$ is Diophantine if and only if $(-\tilde{\eta}_2,\tilde{\eta}_1)$ is Diophantine.
\end{proof}
\end{proof}

We will now show that the winding number appearing in this paper is compatible with vector fields lying on straight lines.

\begin{proposition}\label{compatibilityrotationaltransform}
Let $X$ be a densely non-vanishing vector field such that, for some $f \in C^{\infty}(\mathbb{R}^2/\mathbb{Z}^2)$, numbers $a,b \in \mathbb{R}$ and diffeomorphism $\Phi : S \to \mathbb{R}^2/\mathbb{Z}^2$, we have
\begin{equation*}
\Phi_* X = f\left(a\frac{\partial}{\partial x} + b\frac{\partial}{\partial y}\right).
\end{equation*}
Then, $(a,b) \neq (0,0)$ and $X$ is winding. Considering the standard homology generators $\gamma_1,\gamma_2$ of $H_1(\mathbb{R}^2/\mathbb{Z}^2)$, $X$ has winding number $[(a,b)]$ with respect to the pulled back generators $\Phi^*\gamma_1 = (\Phi^{-1})_*\gamma_1$ and $\Phi^*\gamma_2 = (\Phi^{-1})_*\gamma_2$.
\end{proposition}

\begin{proof}
There exists $p \in S$ with $X|_p \neq 0$. Hence, $(a,b)\neq 0$. Now, in $\mathbb{R}^2/\mathbb{Z}^2$, consider the vector fields $\frac{\partial}{\partial x}$,$\frac{\partial}{\partial y}$ and the closed 1-forms $dx$,$dy$. Then, the form $\tilde{\omega}_0 = bdx - ady$ is closed, non-vanishing, and
\begin{equation*}
\tilde{\omega}_0\left(a\frac{\partial}{\partial x} + b\frac{\partial}{\partial y}\right) = 0.
\end{equation*}
Then, the pull-back $\omega_0 = \Phi^*\tilde{\omega}_0$ is a non-vanishing closed 1-form and satisfies $\omega_0(X) = 0$. Since $\omega_0$ is non-vanishing, $X$ is winding and $[\omega_0] \neq 0$ from Proposition \ref{correctnessrotationaltransform}. Moreover,
\begin{equation*}
\int_{\gamma_1}\tilde{\omega}_0 = b,~ \int_{\gamma_2}\tilde{\omega}_0 = -a.
\end{equation*}
Hence, the winding number of $X$ with respect to $\Phi^*\gamma_1,\Phi^*\gamma_2$ is $[(a,b)]$.
\end{proof}

\subsection{Relationship with linearisability}\label{app:rottranslinear}

In this section, we will prove the Theorem \ref{celestialmechanics} from celestial mechanics~\cite{SternbergCelestial}. To this end, and for completeness, we will first discuss the well-known linearisability techniques employed by Arnold~\cite{arnold-1966,arnold1974asymptotic}. The first of which is focuses on a single 2-torus.

\begin{proposition}\label{Arnoldcommute}
Let $S$ be a compact connected 2-manifold and let $X$ and $Y$ be vector fields which are point-wise independent and $[X,Y] = 0$. Then, there exists a diffeomorphism $\Phi : S \to \mathbb{R}^2/\mathbb{Z}^2$ and numbers $a,b \in \mathbb{R}$ such that
\begin{equation*}
\Phi_*X = a\frac{\partial}{\partial x} + b\frac{\partial}{\partial y}.
\end{equation*}
\end{proposition}

\begin{proof}
In the following, we will use the standard terminology in~\cite[Page 162 and Pages 550 to 552]{Lee}. Let $\psi^X,\psi^Y$ denote the complete flows of $X$ and $Y$ respectively. Define the map $\Psi : \mathbb{R}^2 \times S \to S$ by
\begin{equation*}
\Psi((s,t),p) = (\psi^X_s \circ \psi^Y_t)(p).    
\end{equation*}
Since $[X,Y] = 0$, the flows commute $\psi^X_s \circ \psi^Y_t = \psi^Y_t \circ \psi^X_s$ for $s,t \in \mathbb{R}^2$. Thus, $\Psi$ defines a smooth group action on $S$. Now, for $p \in S$, set $\Psi_p = \Psi(\cdot,p) : \mathbb{R}^2 \to S$. Then, at $0 \in \mathbb{R}^2$, with the curves $C_1,C_2 : \mathbb{R} \to \mathbb{R}^2$ given by $C_1(t) = (t,0)$ and $C_2(t) = (0,t)$ and linear independence of $X|_p$ and $Y|_p$, we see that $T(\Psi_p)|_0$ is invertible. Hence, since $\Psi$ is a group action, we get for all $p \in S$ and $u \in \mathbb{R}^2$ that $T(\Psi_p)|_u$ is invertible. Hence, by the Inverse Function Theorem, $\Psi_p$ is a local diffeomorphism.

In particular, the orbits $\mathbb{R}^2 \cdot p = \Psi_p(\mathbb{R}^2)$ for $p \in S$ form an open disjoint covering of $S$. Thus, by connectedness of $S$, the action $\Psi$ is transitive. That is, $S$ is a homogeneous $\mathbb{R}^2$-space with the action $\Psi$. From now on, fix a point $p \in S$. Considering the isotropy subgroup $\mathbb{R}^2_p = \{u \in \mathbb{R}^2 : \Psi(u, p) = p\}$ of $\mathbb{R}^2$, we have that the map $F : \mathbb{R}^2/\mathbb{R}^2_p \to S$ defined by $F(u+\mathbb{R}^2_p) = \Psi(u,p)$ is an equivariant diffeomorphism~\cite[Theorem 21.18]{Lee}. Then, since the vector fields $\frac{\partial}{\partial x},\frac{\partial}{\partial y}$ in $\mathbb{R}^2$ are translation invariant, they descend to vector fields $\tilde{X}$ and $\tilde{Y}$ in the quotient $\mathbb{R}^2/\mathbb{R}^2_p$.

Now, since $\Psi_p$ is a local diffeomorphism, we also get that $\mathbb{R}^2_p$ is a discrete subgroup of $\mathbb{R}^2$. Thus (see for instance~\cite[Lemma 5.14]{knapp2007advanced}), we must have either $\mathbb{R}^2_p = \mathbb{Z}u$ for some $0 \neq u \in \mathbb{R}^2$ or $\mathbb{R}^2_p = \mathbb{Z}u \oplus \mathbb{Z} v$ for some linearly independent $u,v \in \mathbb{R}^2$. Since the quotient $\mathbb{R}^2/\mathbb{R}^2_p$ is compact, we must have the latter. In particular, there exists an invertible matrix $A \in \text{GL}(2,\mathbb{R})$ such that $A\mathbb{Z}^2  = \mathbb{R}^2_p$ which induces a diffeomorphism $\hat{A} : \mathbb{R}^2/\mathbb{Z}^2 \to \mathbb{R}^2/\mathbb{R}^2_p$. By definition of $F$, one sees that, $F_* \tilde{X} = X$ and $\hat{A}_* \tilde{Y} = Y$. The result then follows with the diffeomorphism $\Phi = (\hat{A} \circ F)^{-1}$.
\end{proof}

Directly related to Arnold's structure Theorems is the following Proposition which is a variant of Proposition \ref{Arnoldcommute} for multiple tori in three dimensions.

\begin{proposition}\label{Arnoldcommute3D}
Consider the product manifold $M = S \times I$ with boundary where $S$ is a compact connected 2-manifold and $I$ is an interval with, for some $\epsilon > 0$, either $I = [0,\epsilon)$ or $I = (-\epsilon,\epsilon)$. Let $X,Y$ be commuting vector fields with $dz(X) = 0 = dz(Y)$ where $z : M \to I$ is projection onto the second factor. Then, there exist smooth functions $a,b : I \to \mathbb{R}$ and a diffeomorphism $\Phi : M \to \mathbb{R}^2/\mathbb{Z}^2 \times I$ such that
\begin{align*}
\Phi_* X &= a(z) \frac{\partial}{\partial x} + b(z) \frac{\partial}{\partial y}.
\end{align*}
\end{proposition}

\begin{proof}
Consider the map $F : \mathbb{R}^2 \times M \to M$ given by
\begin{equation*}
F((s,t),p) = (\psi^X_s \circ \psi^Y_t)(p).
\end{equation*}
Fix a point $p_0 \in S$ and consider the map $G : \mathbb{R}^2 \times I \to M$ given by
\begin{equation*}
G((s,t),z) = F((s,t),(p_0,z)).
\end{equation*}

As seen in the proof of Proposition \ref{Arnoldcommute}, we have for all $z \in I$ that there exists a rank 2 lattice $\Lambda_z \subset \mathbb{R}^2$ such that the map $G_z : \mathbb{R}^2/\Lambda_z \to M$ given by
\begin{equation*}
G_z((s,t)+\Lambda_z) = G((s,t),z)
\end{equation*}
is an embedding on to $(S,z)$ in $M$ satisfying
\begin{align*}
T G_z \circ \frac{\partial}{\partial x} &= X|_{(s,z)}, & T G_z \circ \frac{\partial}{\partial y} &= Y|_{(s,z)}
\end{align*}
where $\frac{\partial}{\partial x}$ and $\frac{\partial}{\partial y}$ are the constant vector fields lowered to the quotient $\mathbb{R}^2/\Lambda_z$.

To turn the $G_z$s into a diffeomorphism with domain $\mathbb{R}^2/\mathbb{Z}^2 \times I$, we must first check that the lattice $\Lambda_z$ smoothly varies with $z \in I$. To this end, consider the vector field frame $(X,Y,Z)$ where $Z = \frac{\partial}{\partial z}$ is the vector field on $M$ induced by the factor $I$. Consider now the induced co-frame $(\alpha,\beta,\gamma)$ of 1-forms to $(X,Y,Z)$ so that
\begin{align*}
\alpha(X) &= 1, & \alpha(Y) &= 0, & \alpha(Z) &= 0,\\
\beta(X) &= 0, & \beta(Y) &= 1, & \beta(Z) &= 0,\\
\gamma(X) &= 0, & \gamma(Y) &= 0, & \gamma(Z) &= 1.
\end{align*}
Note that $\gamma = dz$. Now, fix smooth curves $C_1,C_2 : [0,1] \to S$ which generate the first homology of $S$. For $z \in I$ and $i \in \{1,2\}$ set $C_i^{z} = (C_i,z) : [0,1] \to M$ and
\begin{align*}
v_{1,z} &= \left(\int_{C_1^z}\alpha,\int_{C_1^z}\beta\right), & v_{2,z} = \left(\int_{C_2^z}\alpha,\int_{C_2^z}\beta \right).
\end{align*}
We claim that $(v_{1,z},v_{2,z})$ forms a lattice basis for $\Lambda_z$ for $z \in J$. Indeed, let $z \in I$. Then, we get curves $c_i : [0,1] \to \mathbb{R}^2/\Lambda_z$ induced by the embedding $G_z$ and curves $C_i$ ($i \in \{1,2\}$). We also have that
\begin{align*}
G_z^*\alpha &= dx, & G_z^*\beta &= dy
\end{align*}
where $dx,dy$ are the constant 1-forms lowered to the quotient $\mathbb{R}^2/\Lambda_z$. Now, write $\Lambda_z = \mathbb{Z}u_1 \oplus \mathbb{Z}u_2$ for some $u_1,u_2$ linearly independent in $\mathbb{R}^2$. Then, consider the curves $D_1,D_2 : [0,1] \to \mathbb{R}^2/\Lambda_z$ given by
\begin{align*}
D_1(t) &= t u_1 + \Lambda_z, & D_2(t) &= t u_2 + \Lambda_z.
\end{align*}
Then, $D_1$ and $D_2$ generate the first homology of $\mathbb{R}^2/\Lambda_z$. Hence, since both $([c_1],[c_2])$ and $([D_1],[D_2])$ are generators for the first homology, there exists a matrix $A \in GL(2,\mathbb{Z})$ such that
\begin{equation*}
[c_i] = A_{ij}[D_j].    
\end{equation*}
With this, we get that
\begin{equation*}
v_{z,i}^1 = \int_{C_i^z}\alpha = \int_{c_i}dx = A_{ij}\int_{D_j}dx = A_{ij}u_j^1
\end{equation*}
and similarly $v_{z,i}^2 = A_{ij}u_j^2$. In total, we have $v_{z,i} = A_{ij}u_j$. Hence, since $(u_1,u_2)$ is a generator for $\Lambda_z$, so is $(v_{z,1},v_{z,2})$. Hence, our claim holds.

We will now use our $v_{z,i}$ ($i \in \{1,2\}$) to make a diffeomorphism. To this end, for each $z \in J$, form the matrix
\begin{equation*}
A_z = 
\begin{pmatrix}
v_{z,1} && v_{z,2}
\end{pmatrix}.
\end{equation*}
Then, we have the smooth maps $\hat{A},\hat{B} : \mathbb{R}^2 \times I \to \mathbb{R}^2 \times I$ given by
\begin{equation*}
\hat{A}((s,t),z) = (A_z(s,t),z),~\hat{B}((s,t),z) = (A^{-1}(s,t),z)
\end{equation*}
whereby $\hat{A} \circ \hat{B} = \text{Id} = \hat{B} \circ \hat{A}$ so that in particular, $\hat{A}$ is a diffeomorphism. Moreover, we see that $TG$ is everywhere invertible and that $G(\partial (\mathbb{R}^2 \times J)) = \partial M$. Similarly to the proof of Proposition \ref{Arnoldcommute}, the tangent map $TG$ is invertible everywhere. Hence, in the case of $I = (-\epsilon,-\epsilon)$, the Inverse Function Theorem gives that $G$ is a local diffeomorphism. In the case of $I = [0,\epsilon)$, one may, for instance, globally extend the vector fields $X$ and $Y$ to vector fields $\tilde{X}$ and $\tilde{Y}$ defined on a neighborhood of $S \times (-\epsilon,\epsilon)$ tangent to the compact $S\times\{z\}$ for all $z \in I$, construct the suitable $\tilde{G}$, and apply the Inverse Function Theorem to $\tilde{G}$ using the fact that $S \times \{0\}$ is left invariant by $X$ and $Y$. Hence, in any case, we have the local diffeomorphism
\begin{equation*}
H = G \circ \hat{A}_z : \mathbb{R}^2 \times J \to M.
\end{equation*}
Now, we have the product map
\begin{equation*}
\Pi = \pi \times \text{Id}_{J} : \mathbb{R}^2 \times I \to \mathbb{R}^2/\mathbb{Z}^2 \times I
\end{equation*}
where $\pi : \mathbb{R}^2 \to \mathbb{R}^2/\mathbb{Z}^2$ is the quotient map so that $\Pi$ is an onto local diffeomorphism. With this, since $A_z = (v_{z,1},v_{z,2})$ is a matrix of generators of $\Lambda_z$ for each $z \in I$, then there exists a unique map $\Psi : \mathbb{R}^2/\mathbb{Z}^2 \times I \to M$ such that
\begin{equation*}
\Psi \circ \Pi = H. 
\end{equation*}
Hence, $\Psi$ is a local diffeomorphism. It is also clear that $\Psi$ is bijective. Hence, $\Psi$ is a diffeomorphism and the desired map is $\Phi = \Psi^{-1}$.
\end{proof}

To continue with proving Theorem \ref{celestialmechanics}, we will also use a very special case of Calibi's theorem~\cite{calabi1969intrinsic} on intrinsically harmonic 1-forms, which is the following.

\begin{proposition}\label{Calibiprop}
Let $M$ be a compact oriented manifold. Let $\omega \in \Omega^1(M)$ be a closed and non-vanishing. Then $\omega$ is intrinsically harmonic; that is, there exists a metric $g$ on $M$ such that $\delta \omega = 0$.
\end{proposition}

In particular, Proposition \ref{Calibiprop} and the Riemann-Roch Theorem give the following.

\begin{proposition}
Let $S$ be a oriented 2-torus. Let $\omega \in \Omega^1(S)$ be a closed 1-form on $S$ which is non-zero. Then $\omega$ is intrinsically harmonic if and only if $\omega$ is non-vanishing.
\end{proposition}

\begin{proof}[Proof of Proposition \ref{Calibitorus}]
Suppose that $\omega$ is intrinsically harmonic. Let $g$ be a metric on $S$ for which $\omega$ is harmonic on $(S,g)$. Then, as in Proposition \ref{RiemanniantoRiemann}, there exists a maximal holomorphic atlas $\mathcal{A}$ compatible with the smooth structure on $S$ where the component functions $x,y$ of holomorphic charts satisfy $\star dx = dy$ where $\star$ is the Hodge star of $S$ in $U$. So, $S$ is now a Riemann surface of genus $1$, consider the complex 1-form $W = \omega+i\star\omega$. This 1-form is holomorphic and thus, from the Riemann-Roch Theorem, $W$ is either identically zero or non-vanishing. The other direction is Proposition \ref{Calibiprop}.
\end{proof}

The last result we need to prove Theorem \ref{celestialmechanics} is the following solvability of the cohomological equation~\cite[Proposition 2.6]{KocsardCohomological}.

\begin{proposition}\label{cohomologicalsolvability}
Let $S$ be a 2-torus, $\Phi : S \to \mathbb{R}^2/\mathbb{Z}^2$ be a diffeomorphism and set $X$ such that
\begin{equation*}
\Phi_* X = a\frac{\partial}{\partial x} + b\frac{\partial}{\partial y}
\end{equation*}
where $(a,b)$ is a Diophantine vector. Then, for any $v \in C^{\infty}(S)$, there exists $c \in \mathbb{R}$ and a solution $u \in C^{\infty}(S)$ to the cohomological equation
\begin{equation*}
Z(u) = v - c.
\end{equation*}
That is, setting $\mu = \Phi^*\mu_0$, where $\mu_0$ is the standard volume form on $\mathbb{R}^2/\mathbb{Z}^2$, for any $v \in C^{\infty}(S)$, there exists a solution $u \in C^{\infty}(S)$ to the cohomological equation
\begin{equation*}
X(u) = v,
\end{equation*}
if and only if
\begin{equation*}
\int_S v\mu = 0.
\end{equation*}
\end{proposition}

We will now prove Theorem \ref{celestialmechanics}.

\begin{proof}[Proof of Theorem \ref{celestialmechanics}]
Let $X$ be a non-vanishing vector field on a 2-torus $S$.

\begin{proof}[Statement 1 implies statement 2]
Suppose that $X$ preserves a top-form $\mu \in \Omega^2(S)$. Setting $\omega = i_X\mu$, we have $\omega(X) = 0$ and
\begin{equation*}
d\omega = di_X\mu = i_Xd\mu + di_X\mu = \mathcal{L}_X\mu = 0.
\end{equation*}
\end{proof}

\begin{proof}[Statement 2 implies statement 3]
Suppose that $X$ is winding. Let $\omega$ be a closed non-vanishing 1-form such that $\omega(X) = 0$. Since $\omega$ is non-vanishing, by Theorem \ref{Calibitorus}, there exists an Riemannian metric $g$ on $S$ for which $\omega$ is harmonic on $(S,g)$. Then, consider $\eta = \star \omega$. Then, $\eta$ is closed and the top form $\omega \wedge \eta$ is non-vanishing. Thus, since $X$ is non-vanishing, $i_X(\omega \wedge \eta)$ is non-vanishing. On the other hand,
\begin{equation*}
i_X(\omega \wedge \eta) = \omega(X)\eta - \eta(X)\omega = -\eta(X)\omega.
\end{equation*}
Hence, $\eta(X)$ is non-vanishing. In particular, we have
\begin{equation*}
\omega(X/\eta(X)) = 0,~ \eta(X/\eta(X)) = 1.    
\end{equation*}
From this, consider the unique vector field $Y$ on $S$ such that
\begin{equation*}
\omega(Y) = 1,~ \eta(Y) = 0.    
\end{equation*}
Then, $X$ and $Y$ are point-wise independent and $[X/\eta(X),Y] = 0$. Hence, from Proposition \ref{Arnoldcommute}, $X/\eta(X)$ is linearisable and so $X$ is semi-linearisable.
\end{proof}

\begin{proof}[Statement 3 implies statement 1]
Suppose that $X$ is semi-linearisable. Then, for some $0 < f \in C^{\infty}(S)$, numbers $a,b \in \mathbb{R}$ and diffeomorphism $\Phi : S \to \mathbb{R}^2/\mathbb{Z}^2$ such that
\begin{equation*}
\Phi_*(X/f) = a\frac{\partial}{\partial x} + b\frac{\partial}{\partial y}.
\end{equation*}
Then, considering the standard top-form $\mu_0 \in \Omega^2(\mathbb{R}^2/\mathbb{Z}^2)$, and setting $\mu = \frac{1}{f}\Phi^{*}\mu_0$, we see that $\mathcal{L}_X \mu = 0$.
\end{proof}

Lastly, suppose that $X$ has Diophantine winding number. Then, from the above, together with Proposition \ref{correctnessrotationaltransform} and Proposition \ref{compatibilityrotationaltransform}, we get for some $0 < f \in C^{\infty}(S)$, Diophantine vector $(a,b) \in \mathbb{R}^2$, and diffeomorphism $\Phi : S \to \mathbb{R}^2/\mathbb{Z}^2$, that
\begin{equation*}
\Phi_*(X/f) = a\frac{\partial}{\partial x} + b\frac{\partial}{\partial y}.
\end{equation*}

With this, consider the vector fields $\tilde{X}$ and $Y$ such that,
\begin{align*}
\Phi_*\tilde{X} &= a\frac{\partial}{\partial x} + b\frac{\partial}{\partial y},& \Phi_*Y &= -b\frac{\partial}{\partial x} + a\frac{\partial}{\partial y}.
\end{align*}
Since $(a,b)$ is Diophantine, so is $(-b,a)$. With this, let $u \in C^{\infty}(S)$ and consider the vector field
\begin{equation*}
Y_u = u X + Y.
\end{equation*}
First, notice that $X$ and $Y_u$ are point-wise linearly independent. Moreover,
\begin{align*}
[X,Y_u] &= [f \tilde{X}, Y_u]\\
&= [f \tilde{X}, u f \tilde{X} + Y]\\
&= [f \tilde{X}, u f \tilde{X}] + [f \tilde{X}, Y]\\
&= (f \tilde{X})(u)(f \tilde{X}) + u[f \tilde{X},f \tilde{X}] + (-Y(f)\tilde{X} + f[\tilde{X},Y])\\
&= f^2\tilde{X}(u) X -Y(f)\tilde{X}\\
&= f^2(\tilde{X}(u)-Y(-1/f))\tilde{X}.
\end{align*}
Now, set $\mu = \Phi^*\mu_0$ where $\mu_0$ is the standard top form on $\mathbb{R}^2/\mathbb{Z}^2$. Setting $v = Y(-1/f)$, using Proposition \ref{cohomologicalsolvability} on $Y$, we have that
\begin{equation*}
\int_S Y(-1/f)\mu = \int_S v \mu = 0.
\end{equation*}
Hence, using Proposition \ref{cohomologicalsolvability} on $X$, there exists a solution $u \in C^{\infty}(S)$ to the cohomological equation
\begin{equation*}
\tilde{X}(u) = v =  Y(-1/f).
\end{equation*}
Thus, with this choice of $u$,
\begin{equation*}
[X, Y_u] = f^2(\tilde{X}(u)-Y(-1/f))X = 0.
\end{equation*}
Thus, from Proposition \ref{Arnoldcommute}, $X$ is linearisable.
\end{proof}

\bibliography{Linearisability_of_divergence-free_fields.bib}

\end{document}